\documentclass[10pt]{amsart}
\usepackage[left=2cm,right=2cm,top=2cm,bottom=2cm]{geometry}
\usepackage{amsmath,amssymb,amsthm,amscd,eucal}
\usepackage{enumerate,url}
\usepackage{todonotes}

\newtheorem{theorem}{Theorem}
\newtheorem{corollary}[theorem]{Corollary}
\newtheorem{lemma}[theorem]{Lemma}
\newtheorem{proposition}[theorem]{Proposition}
\newtheorem{remark}[theorem]{Remark}

\let\Im\relax
\DeclareMathOperator{\Im}{Im}
\let\Re\relax
\DeclareMathOperator{\Re}{Re}
\newcommand{\h}{\mathbb{H}}
\newcommand{\p}{\mathbb{P}}
\newcommand{\Q}{\mathbb{Q}}
\newcommand{\R}{\mathbb{R}}
\newcommand{\Z}{\mathbb{Z}}
\newcommand{\E}{\mathcal{E}}
\newcommand{\F}{\mathcal{F}}
\newcommand{\s}{\mathcal{S}}
\newcommand{\T}{\mathcal{T}}
\DeclareMathOperator{\SL}{SL}
\DeclareMathOperator{\PSL}{PSL}
\DeclareMathOperator{\vol}{Vol}
\DeclareMathOperator{\lcm}{lcm}
\DeclareMathOperator{\Id}{Id}
\newcommand{\abs}[1]{\left\vert#1\right\vert}
\newcommand{\jst}[1]{{\bf JST\textbf{#1}} algorithm}

\begin{document}

\title[Certain aspects of holomorphic function theory on some genus zero arithmetic groups]{Certain
aspects of holomorphic function theory on some genus zero arithmetic groups}
\author[J.~Jorgenson]{Jay Jorgenson}
\address{Department of Mathematics, The City College of New York, Convent Avenue at 138th Street,
New York, NY 10031 USA, e-mail: jjorgenson@mindspring.com}
\author[L.~Smajlovi\'c]{Lejla Smajlovi\'c}
\address{Department of Mathematics, University of Sarajevo, Zmaja od Bosne 35, 71\,000 Sarajevo,
Bosnia and Herzegovina, e-mail: lejlas@pmf.unsa.ba}
\author[H.~Then]{Holger Then}
\address{Department of Mathematics, University of Bristol, University Walk, Bristol, BS8\,1TW,
United Kingdom, e-mail: holger.then@bristol.ac.uk}

\begin{abstract}
There are a number of fundamental results in the study of holomorphic function theory associated
to the discrete group $\PSL(2,\Z)$ including the following statements: The ring of holomorphic
modular forms is generated by the holomorphic Eisenstein series of weight four and six, denoted by
$E_{4}$ and $E_{6}$; the smallest weight cusp form $\Delta$ has weight twelve and can be written as
a polynomial in $E_{4}$ and $E_{6}$; and the Hauptmodul $j$ can be written as a multiple of
$E^{3}_{4}$ divided by $\Delta$. The goal of the present article is to seek generalizations of
these results to some other genus zero arithmetic groups $\Gamma_0(N)^+$ with square-free level $N$, which are related to "Monstrous moonshine conjectures".
Certain aspects of our results are generated from extensive computer analysis; as a result, many of the space-consuming results are made available on a publicly
accessible web site. However, we do present in this article specific results for certain low
level groups.
\end{abstract}

\thanks{J.~J.\ acknowledges grant support from NSF and PSC-CUNY grants, and H.~T.\ acknowledges
support from EPSRC grant EP/H005188/1.}

\maketitle

\section{Introduction and statement of results}

Consider the discrete group $\PSL(2, \Z)$ which acts on the upper half plane $\h$.
The quotient space $\PSL(2, \Z)\backslash \h$ has
one cusp which can be taken to be at $i\infty$. Let
$\Gamma_{\infty}$ denote the stabilizer subgroup for the cusp at $i\infty$,
which consists of isometries
$$
\begin{pmatrix} a & b \\ c & d \end{pmatrix} \in \PSL(2,\Z)
$$
with $c=0$. For every integer $k \geq 2$, the
holomorphic Eisenstein series $E_{2k}(z)$ is defined by the absolutely convergent sum
$$
E_{2k}(z) := \sum_{\gamma \in \Gamma_{\infty} \backslash \PSL(2, \Z)}(cz + d)^{-2k}
\quad \text{where $\gamma = \begin{pmatrix} * & * \\ c & d \end{pmatrix}$.}
$$
There is an abundance of important and classical formulae which can be wound back to the
holomorphic Eisenstein series $E_{2k}$. For example, if one defines
$$
G_{2k}(z) := \sum_{(n,m) \in {\textbf Z}^{2} \setminus \{(0,0)\}}(nz + m)^{-2k},
$$
then $E_{2k}(z) = G_{2k}(z)/2\zeta(2k)$ where $\zeta(s)$ is the Riemann zeta function. If
we set $g_{2} = 60 G_{4}$ and $g_{3} = 140G_{6}$, the modular discriminant
$$
\Delta(z) := (2\pi)^{12} e^{2\pi i z}\prod\limits_{n=1}^{\infty}\left(1-e^{2\pi i n z}\right)^{24}
$$
can be written as
\begin{align}\label{Delta-function}
\Delta(z) = g_{2}^{3}(z) - 27 g_{3}^{2}(z) = \frac{1}{1728}(E_{4}^{3}(z) - E_{6}^{2}(z)).
\end{align}
The function $\Delta$ is a weight twelve cusp form with respect to $\PSL(2, \Z)$, meaning it
vanishes as $z$ approaches $i\infty$. It can be shown that no smaller weight cusp form exists.
Furthermore, $\Delta$ is related to the algebraic discriminant of the cubic equation
$y^{2}=4x^{2}-g_{2}x-g_{3}$, in the complex projective coordinates $[x,y,1]$, which defines an
elliptic curve associated to the modular parameter $z$.

All higher weight modular forms associated to $\PSL(2, \Z)$, including Eisenstein series,
can be written in terms of $E_{4}(z)$ and $E_{6}(z)$. For example, the formulae
$E_{8} (z)= E_{4}^{2}(z)$, $E_{10}(z)=E_{4}(z)E_{6}(z)$ and
$$
691 E_{12}(z) = 441 E_{4}(z)^{3} + 250 E_{6}(z)^{2},
$$
are just the beginning of the never ending list of interesting relations which one can write.

Whereas the content of the above discussion is classical, there is a very modern component. The
function
\begin{align}\label{j-invariant}
j(z)=\frac{1728E_4^3(z)}{E_4^3(z) - E_6^2(z)}
\end{align}
is a weight zero modular form on $\PSL(2, \Z)\backslash \h$ which can be viewed as the
biholomorphic function that maps $\PSL(2, \Z)\backslash \h$ onto the Riemann sphere
${\p}^{1}$. If we set $q=e^{2\pi i z}$, then one can expand $j(z)$ as a function of $q$,
namely one has
\begin{align}\label{j_expansion}
j(z) = \frac{1}{q} + 744 + 196884 q + 21493760 q^2 + O(q^3)
\quad
\text{as $q \to 0$.}
\end{align}
In the 1970's, the coefficients in~\eqref{j_expansion} were observed to be related to the sizes
of the irreducible representations of the largest sporadic simple group, which is now known as
``the monster''. The observations
were made precise through the ``Monstrous moonshine conjectures'', some of which are proven in
the celebrated work by Borcherds. We refer the interested reader to \cite{Ga06b} for a thorough
account of the underlying mathematics and physics surrounding the moonshine conjectures as well
as the mathematical history associated to $j(z)$.

Setting to the side the important formulae themselves, one can summarize the above discussion as
the three following points. First, the ring of holomorphic modular forms associated to
$\PSL(2,\Z)$ is generated by $E_{4}$ and $E_{6}$. Second, the smallest weight cusp form $\Delta$
has weight twelve, hence can be written as a polynomial in $E_{4}$ and $E_{6}$. Third, the
Hauptmodul $j$ is equal to a multiple of $E^{3}_{4}$ divided by $\Delta$, hence is a rational
function in $E_{4}$ and $E_{6}$.

The goal of this article is to seek generalizations of the above three statements to certain other
 arithmetic groups related to the "Monstrous moonshine conjectures". Specifically, for any positive integer $N$, let
\begin{align}\label{def:moonsh group}
\Gamma_0(N)^+=\left\{ e^{-1/2} \begin{pmatrix} a & b \\ c & d \end{pmatrix} \in \SL(2,\R):
\quad ad-bc=e, \quad a,b,c,d,e\in\Z, \quad e\mid N,\ e\mid a,\ e\mid d,\ N\mid c \right\}
\end{align}
and let $\overline{\Gamma_0(N)^+} = \Gamma_0(N)^+/\{\pm\Id\}$, where $\Id$ denotes the identity
matrix. Observe that $\PSL(2,\Z) = \overline{\Gamma_0(1)^+}$. It has been shown that there are $43$
square-free integers $N>1$ such that the quotient space
$X_{N} := \overline{\Gamma_0(N)^+}\backslash \h$ has genus zero (see \cite{Cum04}).
Each group has one cusp, which we can always choose to be at $i\infty$.
As stated in the title, the aim of this paper is to present results in the study of the
holomorphic function theory associated to these $43$ spaces.

Let
$$
\eta(z) = e^{2\pi i z/24}\prod\limits_{n=1}^{\infty}\left(1-e^{2\pi i n z}\right)
$$
denote the Dedekind eta function. For any square-free $N$ assume that $N$ has $r$ prime factors,
and let $\lcm(\cdot,\cdot)$ denote the least common multiple function. Let $\sigma(N)$ equal the
sum of divisors of $N$. It was proven in \cite{JST2} that the function
$$
\Delta_N(z) = \left(\prod_{v \mid N} \eta(v z)\right)^{\ell_N},
$$
where
$$
\ell_N = 2^{1-r}\lcm\Big(4,\ 2^{r-1}\frac{24}{(24,\sigma(N))}\Big)
$$
is a weight $k_N=2^{r-1}\ell_N$ modular form on $\overline{\Gamma_0(N)^+}$, vanishing at the cusp
$i\infty$ only. For reasons discussed in \cite{JST2}, we refer to $\Delta_{N}$ as
\textit{the Kronecker limit function on $\overline{\Gamma_0(N)^+}$}.

The main results of the present paper are the following statements which hold true for each
square-free $N$ provided that $X_N$ has genus zero.

\begin{enumerate}\it
\item\label{num:A1} There is an explicitly computed integer $M_{N}$ such that
$\Delta_{N}^{M_{N}}$ is equal to a polynomial $Q_{N}$ in holomorphic Eisenstein series
associated to $\overline{\Gamma_0(N)^+}$;
\item\label{num:A2} The Hauptmodul $j_{N}$ associated to $\overline{\Gamma_0(N)^+}$ is equal
to a rational function whose numerator is a polynomial $P_{N}$ in holomorphic Eisenstein
series and whose denominator is $\Delta_{N}^{M_N}$;
\item\label{num:A3} The polynomials $P_{N}$ and $Q_{N}$ are explicitly computed; hence, we
determine, for each $N$, a finite set $\T^{(N)}$ of holomorphic Eisenstein series such that
any meromorphic form with at most polynomial growth at $i\infty$ can be expressed as a
rational function involving elements of $\T^{(N)}$.
\end{enumerate}\rm
Points~\ref{num:A1} and \ref{num:A2} are direct generalizations of the
formulae~\eqref{Delta-function} and \eqref{j-invariant}.
Point~\ref{num:A3} is a weak generalization of
the result that the ring of holomorphic modular forms associated to $\PSL(2,\Z)$ is generated by
the holomorphic Eisenstein series of weight four and six. For certain small levels, we are able to
compute generators of the ring of holomorphic forms; however, for general $N$, and for future
investigations we plan to undertake, we are content with point~\ref{num:A3} as stated.

The present article is organized as follows. In section~\ref{sec:background} we establish
notation and cite appropriate background material. In particular, we recall the Kronecker limit
formula associated to the non-parabolic Eisenstein series on
$X_{N}=\overline{\Gamma_0(N)^+}\backslash\h$ and an computer
algorithm of \cite{JST2}. In section~\ref{sec:modular forms},
we prove some basic results regarding low weight modular forms for any level $N>1$. In
section~\ref{sec:jst3}, we present results regarding the ring of holomorphic forms for certain small
levels. In section~\ref{sec:examples}, we present a variant of the algorithm of \cite{JST2}
from which we prove that for every square-free $N$, provided that $X_N$ has genus zero, there is an
integer $M_{N}$ such that $\Delta_{N}^{M_{N}}$ can be written as a polynomial in holomorphic
Eisenstein series. Let $j_{N}$ denote the biholomorphic map from $X_{N}$ to the Riemann sphere
$\p^{1}$ which maps $i\infty$ to zero. The algorithm described in section~\ref{sec:examples}
allows us to prove that $j_{N}\Delta_{N}^{M_{N}}$ can be written as a polynomial in holomorpic
Eisenstein series. The data provided by the algorithm is presented in Table~\ref{tab:jst_2_vs_3},
as is a comparison of the results of the original algorithm of \cite{JST2} and the modified
variant thereof. From the algorithm developed in this paper, we are able to determine for
each level $N$ a set of holomorphic Eisenstein series which generate $\T^{(N)}$, the ring of
holomorphic modular forms associated to $X_{N}$; this information is given in
Table~\ref{E generators}. It is important to note that the entries in Table~\ref{E generators}
may not be a minimal set of generators, meaning that for each $N$ there may exist further relations
amongst the sets listed in Table~\ref{E generators}.

As $N$ grows, so does the complexity of the formulae for $\Delta_{N}$ and $j_{N}$. For example,
when $N=17$, our algorithm shows that the five holomorphic Eisenstein series of weights four
through twelve generate $\T^{(17)}$ and $M_{17}=9$, meaning $\Delta_{17}^{9}$ and
$j_{17}\Delta_{17}^{9}$ can be written as a polynomial in these five Eisenstein series. As an
indication of the complexity of the formulae, we present these two examples in
section~\ref{sec:examples}. The formula for $\Delta_{17}$ and $j_{17}$ each occupy
approximately one page.

We note that the Tables 3 and 4a of \cite{CN79} describe, in their notation, how one can express
each Hauptmodul $j_{N}$ in terms of holomorphic forms. In Table~\ref{tab:j known} we
translate the aforementioned data from \cite{CN79}, related to $43$ groups defined by \eqref{def:moonsh group}
with square-free $N$ and genus zero, such that we explicitly write these formulae in terms of
the Dedekind eta function and theta function attached to quadratic forms. By combining our formulae
for $j_{N}$ and the formulae from \cite{CN79} one has the prospect of obtaining further identities
involving holomorphic Eisenstein series and theta functions.

As in \cite{JST2}, the theoretical work developed in this article is supplemented by extensive
computer analysis and, quite frankly, some of the results are not printable. For example, for
$N=119$, the formula for $j_{119}$ from \cite{JST2} occupies nearly $60$ pages.
Nonetheless, in order to disseminate the results obtained by our algorithms, we have posted all
formulae to a web site \cite{JST15URL}.

\section{Background material} \label{sec:background}

\subsection{Holomorphic modular forms}

Let $\Gamma$ be a Fuchsian group of the first kind. Following \cite{Serre73}, we define a weakly
modular form $f$ of weight $2k$ for $k \geq 1$ associated to $\Gamma$ to be a function $f$ which
is meromorphic on $\h$ and satisfies the transformation property
$$
f\left(\frac{az+b}{cz+d}\right) = (cz+d)^{-2k}f(z)
\quad \text{for all $\begin{pmatrix} a&b \\ c&d \end{pmatrix} \in \Gamma$.}
$$

Assume that $\Gamma$ has at least one class of parabolic elements. By transforming coordinates,
if necessary, we
may always assume that the parabolic subgroup of $\Gamma$ has a fixed point at $i\infty$, with
identity scaling matrix. In this situation, any weakly modular form $f$ will satisfy the relation
$f(z+1)=f(z)$, so we can write
$$
f(z) = \sum\limits_{n=-\infty}^{\infty}a_{n}q^{n} \quad \text{where $q = e^{2\pi iz}$.}
$$
If $a_{n} = 0$ for all $n < 0$, then $f$ is said to be holomorphic in the cusp; $f$ is called a
cusp form if $a_{n} = 0$ for all $n \leq 0$. A holomorphic modular form with respect to $\Gamma$
is a weakly modular form which is holomorphic on $\h$ and in all of the cusps of $\Gamma$.

For $\Gamma = \PSL(2, \Z)$, the full modular surface, there is no weight $2$ holomorphic modular
form. Nonetheless, one defines the function $E_2(z)$ by the $q$-expansion
\begin{align}\label{def E_2}
E_2(z) = 1-24 \sum_{n=1}^{\infty} \sigma(n) q^n.
\end{align}
It can be shown that $E_{2}(z)$ transforms according to the formula
\begin{align}\label{transf for E_2}
E_2(\gamma z) = (cz+d)^2 E_2(z) + \frac{6}{\pi i}c (cz+d)
\quad \text{for $\begin{pmatrix} * & * \\ c & d \end{pmatrix} \in \SL(2,\Z)$}.
\end{align}
From this, it is elementary to show that for a prime $p$, the function
$$
E_{2,p}(z) :=\frac{pE_2(pz)-E_2(z)}{p-1}
$$
is weight $2$ holomorphic form associated to the congruence subgroup $\Gamma_0(p)$ of $\SL(2,\Z)$.
The $q$-expansion of $E_{2,p}$ is
$$
E_{2,p}(z)= 1 + \frac{24}{p-1}\sum_{n=1}^{\infty}\sigma(n) (q^n - pq^{pn}).
$$

\subsection{Certain arithmetic groups related to "moonshine"}

For any square-free integer $N$, the subset of $\SL(2,\R)$ defined by~\eqref{def:moonsh group}
is an arithmetic subgroup of $\SL(2,\R)$. As shown in \cite{Cum04}, there are precisely $44$ such groups which have genus
zero and which appear in "Monstrous moonshine conjectures". In this article we will focus on the $43$ genus zero groups for which $N>1$.

We denote by $\overline{\Gamma_0(N)^+} = \Gamma_0(N)^+ / \{\pm\Id\}$ the corresponding subgroup of
$\PSL(2,\R)$. Basic properties of $\Gamma_0(N)^+$, for square-free $N$ are derived in \cite{JST12}
and references therein. In particular, we use that the surface
$X_N=\overline{\Gamma_0(N)^+}\backslash\h$ has exactly one cusp, which can be taken to be at
$i\infty$.

Let $\T^{(N)}$ denote the ring of holomorphic modular forms associated to $X_{N}$, and let
$\T_{2k}^{(N)}$ denote the holomorphic modular forms
of weight $2k$. We will denote the subspace of cusp forms on $X_N$ of weight $2k$ by
$S^{(N)}_{2k}$.

\subsection{Holomorphic Eisenstein series on $\overline{\Gamma_0(N)^+}$}

In the case when $N>1$ is square-free, the holomorphic Eisenstein series associated to
$\overline{\Gamma_0(N)^+}$ are defined for $k \geq 2$ by
$$
E_{2k}^{(N)}(z) := \sum_{\gamma \in \Gamma_{\infty}(N) \backslash \Gamma_0(N)^+}(cz + d)^{-2k}
\quad \text{with $\gamma = \begin{pmatrix} * & * \\ c & d \end{pmatrix}$}
$$
where $\Gamma_{\infty}(N)$ denotes the stabilizer group of the cusp at $i\infty$. In
\cite{JST2} it is proven that $E_{2k}^{(N)}(z)$ may be expressed as a linear combination of forms
$E_{2k}(z)$, the holomorphic Eisenstein series associated to $\PSL(2,\Z)$. Namely, it is known
that
\begin{align}\label{E_k, p proposit fla}
E_{2k}^{(N)}(z)= \frac1{\sigma_k(N)} \sum_{v \mid N}v^k E_{2k}(vz),
\end{align}
where $\sigma_{\alpha}$ denotes the generalized divisor function
$$
\sigma_{\alpha}(m)= \sum_{\delta \mid m} \delta^{\alpha}.
$$
Formula~\eqref{E_k, p proposit fla}, together with a well-known $q$-expansion of classical forms
$E_{2k}$ yields that the $q$-expansion of $E_{2k}^{(N)}$ is given by
\begin{align}\label{E_k^N q-expansion}
E_{2k}^{(N)}(z)=\frac1{\sigma_{k}(N)}\sum_{v | N}v^{k}\left(1-\frac{4k}{B_{2k}}
\sum_{j=1}^\infty \sigma_{2k-1}(j) q^{vj} \right),
\end{align}
where $B_k$ denotes the $k$-th Bernoulli number.

\subsection{Kronecker limit function on $\overline{\Gamma_0(N)^+}$}

Associated to the cusp of $\overline{\Gamma_0(N)^+}$ one has a non-holomorphic
Eisenstein series denoted by $\E^{\text{par}}_\infty(z,s)$ which is defined for
$z \in \h$ and $\Re(s)>1$ by
$$
\E^{\text{par}}_\infty(z,s) = \sum_{\eta\in\Gamma_\infty(N)\backslash\Gamma_0(N)^+}\Im(\eta z)^s.
$$
In \cite{JST2} it is proven that, for any square-free $N$ which has $r$ prime factors, the
parabolic Eisenstein series $\E^{\text{par}}_\infty(z,s)$ admits a Taylor series expansion of
the form
$$
\E^{\text{par}}_\infty(z,s) = 1+ s\cdot \log\left(\sqrt[2^r]{\prod_{v \mid N} \abs{ \eta(vz)}^4}
\cdot \Im(z) \right) + O(s^2), \text{ as } s\to0,
$$
where $\eta(z)$ is Dedekind's eta function associated to $\PSL(2,\Z)$. As stated above, it is
proven that the function
\begin{align}\label{Kronecker_limit}
\Delta_N(z) = \left(\prod_{v \mid N} \eta(v z)\right)^{\ell_N},
\end{align}
where
$$
\ell_N = 2^{1-r}\lcm\Big(4,\ 2^{r-1}\frac{24}{(24,\sigma(N))}\Big)
$$
and $\lcm(\cdot,\cdot)$ denotes the least common multiple function, is weight $k_N=2^{r-1}\ell_N$
modular form on $\overline{\Gamma_0(N)^+}$, vanishing at the cusp $i\infty$ only. We call the
function $\Delta_N(z)$ defined by~\eqref{Kronecker_limit} the
\textit{Kronecker limit function on $\overline{\Gamma_0(N)^+}$}.

\subsection{The algorithm} \label{subsec:jst}

Let $X_{N}=\overline{\Gamma_0(N)^+}\backslash\h$ have genus $g$.
For any positive integer $M$, the function
\begin{align}\label{rational_function_b1}
F_b(z) = \prod_\nu\left(E_{m_\nu}^{(N)}(z)\right)^{b_\nu}\Big/\big(\Delta_N(z)\big)^M
\quad \text{where} \quad
\sum_\nu b_\nu m_\nu = Mk_N
\quad \text{and} \quad
b = (b_1, \ldots)
\end{align}
is a holomorphic modular function on $X_{N}$, meaning a weight
zero modular form with polynomial growth near $i\infty$. The $q$-expansion of $F_{b}$ follows
from substituting the $q$-expansions of $E_k^{(N)}$ and $\Delta_N$.

Let $\s_M$ denote the set of all possible rational functions defined
in~\eqref{rational_function_b1} for all vectors $b=(b_{\nu})$ and $m=(m_{\nu})$ with fixed $M$.
In \cite{JST2}, we implemented the following algorithm, which we refer to as the \jst2.

Choose a non-negative integer $\kappa$. Let $M=1$ and set $\s = \s_1 \cup \s_0$.
\begin{enumerate}
\item Form the matrix $A_{\s}$ of coefficients from the $q$-expansions of all
elements of $\s$, where each element in $\s$ is expanded along a row with each column containing
the coefficient of a power, negative, zero or positive, of $q$. The expansion is recorded out to
order $q^{\kappa}$.
\item Apply Gauss elimination to $A_{\s}$, thus producing a matrix $B_{\s}$ which is
in row-reduced echelon form.
\item\label{num:B3} Implement the following decision to determine if the algorithm is complete:
If the $g+1$ lowest non-trivial rows at the bottom of $B_{\s}$ correspond to $q$-expansions whose
lead terms have precisely $g$ gaps, meaning zero coefficients, in the set
$\{q^{-1}, \ldots, q^{-2g}\}$, then the algorithm is completed. If the indicator to stop fails,
then replace $M$ by $M+1$, $\s$ by $\s_M \cup \s$ and reiterate the algorithm.
\end{enumerate}

If $g=0$, then the algorithm stops if the lowest non-trivial row at the bottom of $B_{\s}$ has a
$q$-expansion which begins with $q^{-1}$. We also denote by $M_N$ the value of $M$ for the group of level $N$ at
which Step~\ref{num:B3} shows that the algorithm is completed.

As stated in \cite{JST2}, the rationale for the stopping decision in Step~\ref{num:B3} above is
based on two
ideas, one factual and one hopeful. First, the Weierstrass gap theorem states that for any point
$P$ on a compact Riemann surface there are precisely $g$ gaps in the set of possible orders from
$1$ to $2g$ of functions whose only pole is at $P$. Second, for any genus, the assumption which
is hopeful is that the function field is generated by the set of holomorphic modular functions
defined in~\eqref{rational_function_b1}, which is related to the question of whether the field of
meromorphic modular forms on $\overline{\Gamma_0(N)^+}$ is generated by holomorphic Eisenstein
series and the Kronecker limit function. The latter assumption is not obvious, and, indeed, the
assumption itself is equivalent to the statement that the rational function field on $X_{N}$ is
generated by the holomorphic Eisenstein series. As it turned out, for all groups
$\overline{\Gamma_0(N)^+}$ that we have studied so far, which includes all groups of
genus zero, genus one, genus two, and genus three, the algorithm stopped. Therefore, we conclude
that, in particular, the rational function field associated to all genus zero groups $\overline{\Gamma_0(N)^+}$ is
generated by a finite set of holomorphic Eisenstein series.

We described the algorithm with choice of an arbitrary $\kappa$ and $g$. For reasons of
efficiency, we initially selected $\kappa$ to be zero, so that all coefficients for $q^\nu$
for $\nu \leq \kappa$ are included in $A_{\s}$. In \cite{JST2}, it is shown that for each $N$,
there is an explicitly computable $\kappa=\kappa_{N}$ such that if a modular form associated to
$\overline{\Gamma_0(N)^+}$ has integral coefficients in its $q$-expansion out to $q^{\kappa_N}$,
then
all remaining coefficients are also integral. The list of $\kappa_{N}$ for square-free levels $N$
provided that $\overline{\Gamma_0(N)^+}$ has genus zero is given in Table 1 of \cite{JST2}.
In the implementation of the above algorithm, both in the present article and in \cite{JST2}, the
value of $\kappa$ was finally increased to $\kappa_{N}$.

In the present article, we implemented a slight variant of the above algorithm, which we refer to
as the \jst3. The difference between the {\bf JST2} and the \jst3 is the following action should
the decision in Step~\ref{num:B3} fail.

\textit{Replace $M$ by $M+1$, $\s$ by $\s_M$, and reiterate the algorithm.}
\\
In other words, the \jst3 studies the $q$-expansions of the space of
rational functions of the form~\eqref{rational_function_b1} with a fixed denominator
$\big(\Delta_N(z)\big)^M$. Should the \jst3 successfully complete, then
the row in $B_{\s}$ with $q$-expansion beginning with $q^{-1}$ would correspond to a
formula for $j_{N}$ with denominator $\big(\Delta_N(z)\big)^M$ and numerator given as a
polynomial in Eisenstein series. Furthermore, any lower row in $B_{\s}$ would correspond to a
$q$-expansion beginning with $q^{0}$, which would yield, upon clearing the denominator, a
formula for $\big(\Delta_N(z)\big)^M$ in terms of Eisenstein series.

As we will report below, the \jst3 has successfully completed for all genus zero
groups $\overline{\Gamma_0(N)^+}$ with square-free level $N$.

\section{Modular forms on surfaces $X_N$} \label{sec:modular forms}

From Proposition 7, page II-7, of \cite{SCM66}, we immediately obtain the following formula which
relates the number of zeros of a modular form, counted with multiplicity, with its weight and
volume of $X_{N}$.

\begin{lemma}\label{prop sum over zeros}
Let $f$ be a modular form on $X_N$ of weight $2k$, not identically zero. Let $\F_N$
denote the fundamental domain of $X_N$ and let $v_z(f)$ denote the order of zero $z$ of $f$. Then,
\begin{align}\label{zeros f-la}
k \frac{\vol(X_N)}{2\pi}=v_{i\infty}(f) + \sum_{e \in \E_N} \frac{1}{n_e} v_{e}(f)
+ \sum_{z\in \F_N \setminus \E_N} v_z(f),
\end{align}
where $\E_N$ denotes the set of elliptic points in $\F_N$ and $n_e$ is the
order of the elliptic point $e\in \E_N$.
\end{lemma}

Lemma~\ref{prop sum over zeros} enables us to deduce the following proposition.

\begin{proposition}
Let $N$ be a square-free number such that the surface $X_N$ has genus zero. Then, there are no
weight two holomorphic modular forms on $X_N$.
\end{proposition}

\begin{proof}
From the tables in \cite{Cum04}, one determines that all genus zero groups
$\overline{\Gamma_0(N)^+}$, for a square-free $N$ have a at most one elliptic point of order three,
four or six and a various number of order two elliptic points. Let $e_{N}(2)$ denote the number
of order two elliptic points on $X_{N}$, and let $n_N \in\{3,4,6\}$ denote the order of the
additional elliptic point on $X_N$. Since all surfaces $X_N$ have one cusp and genus zero,
the Gauss-Bonnet formula for the volume of the surface $X_N$ becomes
\begin{align}\label{volume of X_N}
\frac{\vol(X_N)}{2\pi}= \frac{1}{2}e_{N}(2) + \left(1-\frac{1}{n_N} \right)\delta(N) -1,
\end{align}
where $\delta(N)$ is equal to one if $X_N$ has an elliptic point of order different from two and
zero otherwise.

For an arbitrary, square-free $N$ and $e \mid N$, the elliptic element of
$\overline{\Gamma_0(N)^+}$ is of the form
\begin{align}\label{ell element}
\begin{pmatrix} a\sqrt{e} & b/\sqrt{e} \\ (cN)/\sqrt{e} & d\sqrt{e} \end{pmatrix},
\end{align}
where $a,b,c,d, \in \Z$ are such that $\vert (a+d)\sqrt{e}\vert <2$ and $ade - (bcN)/e=1$.
The first condition implies that either $a+d=0$ or $\vert a+d\vert=1$ and $e \in\{1,2,3\}$.

If $\vert a+d\vert=1$, then $d=\pm 1 -a$, hence
$$
\begin{pmatrix} a\sqrt{e} & b/\sqrt{e} \\ (cN)/\sqrt{e} & d\sqrt{e} \end{pmatrix}^2
= \begin{pmatrix} \pm ae -1 & \pm b \\ \pm cN & \pm ae -1 \end{pmatrix} \neq \pm \Id,
$$
for any choice of $a,b,c \in \Z$ such that $a(\pm 1 -a)e - (bcN)/e=1$. Therefore, there are no
order two elements in $\overline{\Gamma_0(N)^+}$ such that $\vert a + d \vert =1$.

On the other hand, if $a+d=0$, then $-a^2e - (bcN)/e=1$, hence
$$
\begin{pmatrix} a\sqrt{e} & b/\sqrt{e} \\ (cN)/\sqrt{e} & d\sqrt{e} \end{pmatrix}^2
= \begin{pmatrix} -1 & 0 \\ 0 & -1 \end{pmatrix}.
$$
In other words, any elliptic element~\eqref{ell element} of $\overline{\Gamma_0(N)^+}$ has order
two if and only if $a+d=0$. Let
$$
\eta=\begin{pmatrix} a\sqrt{e} & b/\sqrt{e} \\ (cN)/\sqrt{e} & -a\sqrt{e} \end{pmatrix}
$$
denote arbitrary elliptic element of $\overline{\Gamma_0(N)^+}$ of order two, and
let $z_{\eta}$ be its
fixed point. Solving the equation $\eta(z_{\eta})= z_{\eta}$ leads to the conclusion that
$(z_{\eta}cN/\sqrt{e} -a\sqrt{e})^2 = -1 $.

Assume $f_{2,N}$ is a holomorphic modular form on $X_N$ of weight $2$. By the transformation
rule, we have that
$$
f_{2,N}(z_{\eta})=f_{2,N}(\eta z_{\eta})= (-1) f_{2,N}(z_{\eta}),
$$
hence $z_{\eta}$ is vanishing point of $f_{2,N}$. Since this holds true for any order two
elliptic element of $\overline{\Gamma_0(N)^+}$, we conclude that all order two elliptic points
of $X_N$ are vanishing points of $f_{2,N}$. Applying Lemma~\ref{prop sum over zeros} to $f_{2,N}$,
we arrive at the inequality
$$
\frac{\vol(X_N)}{2\pi} \geq \frac{1}{2}e_{N}(2),
$$
which contradicts~\eqref{volume of X_N}. Therefore, there are no weight two holomorphic modular
forms on $X_N$.
\end{proof}

Though there are no weight two holomorphic forms on $\Gamma_0(N)^+$, we may construct forms that
transform almost like a weight two form, up to an order two character.

\begin{proposition}\label{prop: weight two fom with character}
Let $N=p_1 \cdots p_r$ be a square-free positive integer. Let $\mu(\nu)$ denote the
M\"obius function and $E_{2}$ the series defined in~\eqref{def E_2}. Then the holomorphic function
$$
E_{2,N}(z):=\frac{(-1)^r}{\varphi(N)} \sum_{v \mid N} \mu(v) v E_2(vz)
$$
satisfies the transformation rule
$$
E_{2,N}(\gamma_e z) = \mu(e)(c\frac{N}{\sqrt{e}} z + d \sqrt{e})^2 E_{2, N}(z)
$$
for any
$$
\gamma_e=\begin{pmatrix} a\sqrt{e} & b/\sqrt{e} \\ (cN)/\sqrt{e} & d\sqrt{e} \end{pmatrix}
\in \Gamma_0(N)^+.
$$
\end{proposition}

\begin{proof}
Choose and fix any $e\mid N$. For any $v \mid N$, let $(e,v)$ denote the greatest common divisor
of $e$ and $v$. Then, using the transformation formula~\eqref{transf for E_2} for $E_2$, it is
easy to deduce that
$$
v E_2(v(\gamma_e z)) = \frac{ev}{(e,v)^2}(c\frac{N}{\sqrt{e}} z + d \sqrt{e})^2
E_2\left(\frac{ev}{(e,v)^2}z \right)+ \frac{6}{\pi i}cN \left(c\frac{N}{e}z + d \right).
$$
Since $N$ is square-free with $r$ prime factors, it is easy to see that
$$
\sum_{v \mid N} \mu(v) \frac{6}{\pi i}cN \left(c\frac{N}{e}z + d \right) =
\frac{6}{\pi i}cN \left(c\frac{N}{e}z + d \right) \sum_{j=1}^r \binom{r}{j} (-1)^j=0,
$$
hence
$$
\sum_{v \mid N} \mu(v) v E_2(v(\gamma_e z))= \sum_{v \mid N} \mu(v)\frac{ev}{(e,v)^2}
(c\frac{N}{\sqrt{e}} z + d \sqrt{e})^2 E_2\left(\frac{ev}{(e,v)^2}z \right).
$$
We claim that $\{ v : v \mid N\} = \{\frac{ev}{(e,v)^2} : v \mid N \}$, which is easily deduced
by induction in $r$. Furthermore, when $e$ has an even number of prime factors, the parity of the
number of factors of $\frac{ev}{(e,v)^2}$ remains the same as the parity of the number of factors
of $v$, while when $e$ has odd number of factors, the parity changes, meaning that
$\mu(v) = \mu(e) \mu(\frac{ev}{(e,v)^2})$. Therefore
$$
\sum_{v \mid N} \mu(v) v E_2(v(\gamma_e z))
= \mu(e) (c\frac{N}{\sqrt{e}} z + d \sqrt{e})^2 \sum_{v \mid N} \mu(v) v E_2(vz)
$$
and the proof is complete.
\end{proof}

\begin{proposition}
The smallest even integer $k_N$ such that there exists a weight $k_N$ cusp form $f_N$ vanishing
only at the cusp $i\infty$ is given by the formula
\begin{align}\label{k_N formula}
k_N=\lcm(4, 2^{r-1}\frac{24}{(24, \sigma(N))})
\end{align}
where $\lcm$ denotes the least common multiple and $(\cdot, \cdot)$ stands for the greatest
common divisor.
\end{proposition}

\begin{proof}
From \cite{JST2}, one has that the volume of the surface $X_N$ is given by
\begin{align}\label{volume formula}
\vol(X_N)= \frac{\pi \sigma(N)}{6 \cdot2^{r-1}},
\end{align}
where $r$ is the number of (distinct) prime factors of $N$.
By combining~\eqref{volume formula} with \eqref{zeros f-la}, we have that
$$
k_N \cdot \frac{\sigma(N)}{24 \cdot2^{r-1}} = v_{i\infty}(f_N),
$$
hence $2^{r-1}\frac{24}{(24, \sigma(N))} \mid k_N$.

On the other hand, the cusp form $f_N$ does not vanish at order two elliptic points. As proven
above, every surface $X_N$ for a square-free $N$ has at least one order two elliptic point that
is a fixed point of the Atkin-Lehner involution $\tau_N: z \mapsto -1/(Nz)$). Since
$$
f_N(\tau_N(i/\sqrt{N}))=f_N(i/\sqrt{N}) = (i)^{k_N}f_N(i/\sqrt{N}),
$$
it folows that $4 \mid k_N$. The smallest $k_N$ divisible by both $4$ and
$2^{r-1}\frac{24}{(24, \sigma(N))}$ is given by~\eqref{k_N formula}.
Therefore, the proof is complete.
\end{proof}

The above proposition, together with Theorem 16 form \cite{JST2} yields the following corollary.

\begin{corollary}
Let $\ell_N= 2^{1-r} k_N$, where $k_N$ is given by~\eqref{k_N formula}. Then, the function
$$
\Delta_N(z): = \left( \prod_{v \mid N} \eta(v z)\right) ^{\ell_N}
$$
is the smallest weight cusp form on $X_N$ vanishing at the cusp only. Furthermore, the order of
vanishing of $\Delta_N$ at the cusp is given by
$$
v_{i\infty} (\Delta_N)=\frac{\sigma(N)\ell_N}{24} = k_N \cdot \frac{\sigma(N)}{24 \cdot2^{r-1}}
$$
\end{corollary}

\noindent The next proposition determines the smallest weight $\widetilde{k}_N$ for square-free
$N$ such that the space $S_{\widetilde{k}_N}^{(N)}$ is not empty.

\begin{proposition}
Let $N=p_1 \cdots p_r$ be a square-free positive integer where $N>1$. Then, the smallest even
integer $\widetilde{k}_N$ such that there exist a weight $\widetilde{k}_N$ cusp form on a genus
zero surface $X_N$ is equal to $8$, for $N\in\{2,3\}$ and equal to $4$ for all other $N$.
\end{proposition}
\begin{proof}

When $N=2$, it is immediate that $k=8$ is the smallest number such that
$k\cdot \frac{\vol(X_2)}{4\pi} \geq 1$.
Since $\Delta_2$ is weight $8$ cusp form, the assertion is proven when $N=2$.

When $N=3$, $k=6$ is the smallest number such that $k\cdot \frac{\vol(X_3)}{4\pi} \geq 1$.
However, if there exists a weight $6$ cusp form on $X_3$, this cusp form also vanishes at order
two elliptic point $e_2$ of $X_3$. Therefore, the right hand side of the formula~\eqref{zeros f-la}
is at least $3/2$, while the left hand side of the same formula
with $k=6$ is equal to $1$, which yields a contradiction. This shows that $8$ is the
smallest possible weight of cusp form on $X_3$. An example of weight eight cusp form on $X_3$
is $E_8^{(3)} - (E_4^{(3)})^2$, so the case when $N=3$ is complete.

When $N\geq 5$ we can construct the weight four cusp form on $X_N$, whether or not the genus is
zero, as follows. Let
$$
E_{4,N}(z):=\left( E_{2,N} (z)\right)^2.
$$
From Proposition~\ref{prop: weight two fom with character}, it is immediate that $E_{4,N}$
is weight four holomorphic form on $\overline{\Gamma_0(N)^+}$.
Recall that for a square-free $N$ with $r$ prime factors we have the formula
$$
\varphi(N)= (-1)^r \sum_{v \mid N} v\mu(v).
$$
The $q$-expansion~\eqref{def E_2} implies that $E_{2,N} (z)$ is normalized so that its
$q$-expansion has a leading coefficient equal to one. Therefore, the difference
$$
\widetilde{\Delta}_N(z):= E_4^{(N)}(z) - E_{4,N}(z)
$$
is a weight four cusp form. By computing the $q$-expansion of $E_{4,N}$, we deduce that
the term multiplying $q$ in the $q$-expansion of $E_{4,N}(z)$ is $\frac{48}{\varphi(N)}$,
while the term multiplying $q$ in the $q$-expansion of $E_4^{(N)}(z)$ is equal to
$\frac{240}{1+N^2}$. In other words, for square-free $N \notin\{2,3\}$, we have the expansion
$$
\widetilde{\Delta}_N(z) = 48\left( \frac{1}{\varphi(N)} - \frac{5}{N^2+1}\right)q + \ldots.
$$
The leading coefficient is non-zero whenever $N \geq 5$, hence $\widetilde{\Delta}_N(z)$ is
a weight four cusp form on $X_N$.
\end{proof}

\section{Expressing the Hauptmodul in terms of Eisenstein series} \label{sec:jst3}

In this section we discuss the main results of this article.

\begin{theorem}\label{thm:generators}
For any square-free $N \ge 1$ such that the surface $X_N$ has genus zero, there exist
effectively computable integers $M_N$ and $m_N$, and explicitly computable polynomials
$P_N (x_1, \ldots, x_{m_N-1})$ and $Q_N(x_1, \ldots, x_{m_N-1})$ in $m_N-1$ variables with integer
coefficients such that the Hauptmodul $j_N$ can be written as
$$
j_N(z)=\frac{P_N(E_4^{(N)}, \ldots, E_{2m_N}^{(N)})}{Q_N(E_4^{(N)}, \ldots, E_{2m_N}^{(N)})}
$$
and the Kronecker limit function can be written as
$\Delta_{N}^{M_{N}} = Q_N(E_4^{(N)}, \ldots, E_{2m_N}^{(N)})$.
\end{theorem}

\begin{proof}
The result follows, because for each square-free level $N$, provided that $X_N$ has genus zero,
the \jst3 terminates in finite time. As stated, the computer code as well as the output is available on web site \cite{JST15URL}. In the space below, let us describe
in further detail the output of the computational algorithm. We remind the reader that the
{\bf JST2} and the \jst3 are described in section~\ref{subsec:jst}.

After Gauss elimination, one of the $q$-expansions has a pole of order $1$. This is the
Hauptmodul, see section~\ref{sec:examples} for explicit examples.
Keeping track of the linear algebra, we have an exact expression for the Hauptmodul as a
linear combination of holomorphic modular functions~\eqref{rational_function_b1} with rational
coefficients. In other words,
$$
j_N(z)= \frac{1}{(\Delta_N(z))^{M_N}} \sum_{b} C_b
\cdot \left( \prod_\nu\left(E_{m_\nu}^{(N)}(z)\right)^{b_\nu} \right),
$$
where the sum is taken over all $b= (b_1, \ldots)$ such that $\sum_{\nu} b_{\nu} m_{\nu} =k_N M_N$,
where $M_N$ is given in the right column of Table~\ref{tab:jst_2_vs_3} and $C_b \in \Q $.

There is also a $q$-expansion which is equal to the constant $1$. Again, by keeping
track of the linear algebra, we have an exact expression for the constant $1$ as
$$
1=\frac{1}{(\Delta_N(z))^{M_N}} \sum_{b} D_b
\cdot \left( \prod_\nu\left(E_{m_\nu}^{(N)}(z)\right)^{b_\nu} \right),
$$
where the sum is taken over the same set of $b$ as above and $D_B \in \Q$.

By the design of the \jst3, this exact expression can easily be solved for the $M_N$-th power of
the Kronecker limit function, showing that
$$
j_N(z)= \frac{\sum_{b} C_b
\cdot \left( \prod_\nu\left(E_{m_\nu}^{(N)}(z)\right)^{b_\nu} \right)}{
\sum_{b} D_b \cdot \left( \prod_\nu\left(E_{m_\nu}^{(N)}(z)\right)^{b_\nu} \right)}.
$$
After multiplication of both numerator and denominator with the least common multiple of the
denominators of the numbers $C_b$ and $D_b$, we deduce the statement of the theorem.
\end{proof}

\begin{remark}\rm
The polynomials $P_N$ and $Q_N$ appearing in Theorem~\ref{thm:generators} are weighted homogeneous
in the sense that there exists an integer $M_N$ such that the coefficient of the term
$(x_1)^{\alpha_1} \cdots (x_{m_N-1})^{\alpha_{m_N-1}}$ is non-zero only if
$4\alpha_1 + 6\alpha_2 + \ldots + 2m_N \alpha_{m_N-1} = k_N M_N$, where $k_N$ is the weight
of the Kronecker limit function $\Delta_N$.
\end{remark}

\begin{remark}\rm
Table~\ref{tab:jst_2_vs_3} provides the data regarding the performance of the {\bf JST2}
and {\bf JST3} algorithms. More precisely, the first columns of data in Table~\ref{tab:jst_2_vs_3}
lists, for each level $N$ provided that $X_{N}$ has genus zero, the weight $k_{N}$ of the Kronecker
limit function and the integer $\kappa_N$. To
recall, it is shown in \cite{JST2} if the $q$-expansion of a holomorphic modular form has integer
coefficients out to $q^{\kappa_N}$, then all further coefficients are also integral. The columns of
data in Table~\ref{tab:jst_2_vs_3} under the heading \jst2 lists the integer $M$ such that the
\jst2 stops, together with the $q$-expansions which are used in the Gauss elimination algorithm
as well as the order of the largest pole at $i\infty$ amongst the rational functions considered.
The columns of data in Table~\ref{tab:jst_2_vs_3} under the heading \jst3 present
similar information.
\end{remark}

\begin{table}
\caption{\label{tab:jst_2_vs_3}Performance of the {\bf JST2} and the \jst3.
For all genus zero groups $\overline{\Gamma_0(N)^+}$ we list the level $N$, the weight $k_{N}$ of the Kronecker
limit function, the value of $\kappa_N$ in the proof of integrality \cite{JST2} (left);
the level $N$, the number of iterations $M$, the number of equations,
and the largest order of pole for the \jst2 (middle) and similar for the \jst3 (right).}
\begin{minipage}{.2\textwidth}
\begin{center}
\mbox{} \\ \mbox{} %
\end{center}
$$
\begin{array}{ccccccc}
N && k_{N} & \kappa_N \\ \hline
1 && 12 & 19 \\
2 && 8 & 47 \\
3 && 12 & 48 \\
5 && 4 & 19 \\
6 && 4 & 60 \\
7 && 12 & 19 \\
10 && 8 & 75 \\
11 && 4 & 19 \\
13 && 12 & 19 \\
14 && 4 & 47 \\
15 && 4 & 96 \\
17 && 4 & 19 \\
19 && 12 & 19 \\
21 && 12 & 53 \\
22 && 4 & 47 \\
23 && 4 & 19 \\
26 && 8 & 47 \\
29 && 4 & 19 \\
30 && 4 & 127 \\
31 && 12 & 19 \\
33 && 4 & 48 \\
34 && 8 & 47 \\
35 && 4 & 19 \\
38 && 4 & 47 \\
39 && 12 & 48 \\
41 && 4 & 19 \\
42 && 4 & 108 \\
46 && 4 & 47 \\
47 && 4 & 19 \\
51 && 4 & 48 \\
55 && 4 & 19 \\
59 && 4 & 19 \\
62 && 4 & 47 \\
66 && 4 & 60 \\
69 && 4 & 48 \\
70 && 4 & 181 \\
71 && 4 & 19 \\
78 && 4 & 81 \\
87 && 4 & 48 \\
94 && 4 & 47 \\
95 && 4 & 19 \\
105 && 4 & 181 \\
110 && 4 & 89 \\
119 && 4 & 19 \\
\end{array}
$$
\end{minipage}
\quad
\begin{minipage}{.3\textwidth}
\begin{center}
\jst2
\end{center}
$$
\begin{array}{ccccccc}
N && M & \#\{eqs\} & \text{pole} \\ \hline
1 && 1 & 5 & 1 \\
2 && 1 & 3 & 1 \\
3 && 1 & 5 & 2 \\
5 && 1 & 2 & 1 \\
6 && 1 & 2 & 1 \\
7 && 1 & 5 & 4 \\
10 && 2 & 10 & 6 \\
11 && 3 & 8 & 6 \\
13 && 2 & 26 & 14 \\
14 && 3 & 8 & 6 \\
15 && 3 & 8 & 6 \\
17 && 4 & 15 & 12 \\
19 && 3 & 114 & 30 \\
21 && 2 & 26 & 16 \\
22 && 4 & 15 & 12 \\
23 && 5 & 27 & 20 \\
26 && 3 & 31 & 21 \\
29 && 6 & 48 & 30 \\
30 && 4 & 15 & 12 \\
31 && 4 & 434 & 64 \\
33 && 5 & 27 & 20 \\
34 && 3 & 31 & 27 \\
35 && 5 & 27 & 20 \\
38 && 5 & 27 & 25 \\
39 && 3 & 114 & 42 \\
41 && 7 & 82 & 49 \\
42 && 5 & 27 & 20 \\
46 && 6 & 48 & 36 \\
47 && 8 & 137 & 64 \\
51 && 6 & 48 & 36 \\
55 && 6 & 48 & 36 \\
59 && 9 & 225 & 90 \\
62 && 7 & 82 & 56 \\
66 && 6 & 48 & 36 \\
69 && 7 & 82 & 56 \\
70 && 6 & 48 & 36 \\
71 && 10 & 362 & 120 \\
78 && 6 & 48 & 42 \\
87 && 7 & 82 & 70 \\
94 && 8 & 137 & 96 \\
95 && 7 & 82 & 70 \\
105 && 7 & 82 & 56 \\
110 && 7 & 82 & 63 \\
119 && 8 & 137 & 96 \\
\end{array}
$$
\end{minipage}
\quad
\begin{minipage}{.3\textwidth}
\begin{center}
\jst3
\end{center}
$$
\begin{array}{ccccccc}
N && M_N & \#\{eqs\} & \text{pole} \\ \hline
1 && 1 & 4 & 1 \\
2 && 1 & 2 & 1 \\
3 && 1 & 4 & 2 \\
5 && 3 & 4 & 3 \\
6 && 3 & 4 & 3 \\
7 && 2 & 21 & 8 \\
10 && 2 & 7 & 6 \\
11 && 9 & 88 & 18 \\
13 && 3 & 88 & 21 \\
14 && 6 & 21 & 12 \\
15 && 5 & 12 & 10 \\
17 && 9 & 88 & 27 \\
19 && 4 & 320 & 40 \\
21 && 2 & 21 & 16 \\
22 && 6 & 21 & 18 \\
23 && 15 & 1039 & 60 \\
26 && 4 & 55 & 28 \\
29 && 15 & 1039 & 75 \\
30 && 6 & 21 & 18 \\
31 && 5 & 1039 & 80 \\
33 && 8 & 55 & 32 \\
34 && 4 & 55 & 36 \\
35 && 7 & 34 & 28 \\
38 && 10 & 137 & 50 \\
39 && 3 & 88 & 42 \\
41 && 21 & 8591 & 147 \\
42 && 7 & 34 & 28 \\
46 && 14 & 708 & 84 \\
47 && 27 & 56224 & 216 \\
51 && 11 & 210 & 66 \\
55 && 8 & 55 & 48 \\
59 && 33 & 310962 & 330 \\
62 && 18 & 3094 & 144 \\
66 && 8 & 55 & 48 \\
69 && 14 & 708 & 112 \\
70 && 8 & 55 & 48 \\
71 && 39 & 1512301 & 468 \\
78 && 9 & 88 & 63 \\
87 && 17 & 2167 & 170 \\
94 && 26 & 41646 & 312 \\
95 && 11 & 210 & 110 \\
105 && 9 & 88 & 72 \\
110 && 9 & 88 & 81 \\
119 && 10 & 137 & 120 \\
\end{array}
$$
\end{minipage}
\vspace*{4cm} %
\end{table}

\begin{remark}\rm
Table~\ref{E generators} provides a list of the holomorphic Eisenstein series
$E_{m_{\nu}}^{(N)}$ which appear in the expression for the Hauptmodul $j_N$ cited in
Theorem~\ref{thm:generators}. For all levels, the highest weight Eisenstein series has
weight $26$.
\end{remark}

\begin{table}
\caption{\label{E generators}Finite sets of Eisenstein series which include the generators of the
holomorphic Eisenstein series on groups $\overline{\Gamma_0(N)^+}$ of genus zero. Listed are level and finite set.}
\small
$$
\begin{array}{ccc}
N && \text{finite set} \\[1ex]
1 && E_{4}^{(1)},E_{6}^{(1)} \\
2 && {E_{4}^{(2)}},{E_{6}^{(2)}},{E_{8}^{(2)}} \\
3 && {E_{4}^{(3)}},{E_{6}^{(3)}},{E_{12}^{(3)}} \\
5 && {E_{4}^{(5)}},{E_{6}^{(5)}},{E_{8}^{(5)}},{E_{12}^{(5)}} \\
6 && {E_{4}^{(6)}},{E_{6}^{(6)}},{E_{8}^{(6)}},{E_{12}^{(6)}} \\
7 && {E_{4}^{(7)}},{E_{6}^{(7)}},{E_{8}^{(7)}},{E_{10}^{(7)}},{E_{12}^{(7)}} \\
10 && {E_{4}^{(10)}},{E_{6}^{(10)}},{E_{8}^{(10)}},{E_{10}^{(10)}},{E_{12}^{(10)}},{E_{16}^{(10)}} \\
11 && {E_{4}^{(11)}},{E_{6}^{(11)}},{E_{8}^{(11)}},{E_{10}^{(11)}},{E_{12}^{(11)}} \\
13 && {E_{4}^{(13)}},{E_{6}^{(13)}},{E_{8}^{(13)}},{E_{10}^{(13)}},{E_{12}^{(13)}} \\
14 && {E_{4}^{(14)}},{E_{6}^{(14)}},{E_{8}^{(14)}},{E_{10}^{(14)}},{E_{12}^{(14)}} \\
15 && {E_{4}^{(15)}},{E_{6}^{(15)}},{E_{8}^{(15)}},{E_{10}^{(15)}},{E_{12}^{(15)}},{E_{14}^{(15)}},{E_{16}^{(15)}} \\
17 && {E_{4}^{(17)}},{E_{6}^{(17)}},{E_{8}^{(17)}},{E_{10}^{(17)}},{E_{12}^{(17)}} \\
19 && {E_{4}^{(19)}},{E_{6}^{(19)}},{E_{8}^{(19)}},{E_{10}^{(19)}},{E_{12}^{(19)}} \\
21 && {E_{4}^{(21)}},{E_{6}^{(21)}},{E_{8}^{(21)}},{E_{10}^{(21)}},{E_{12}^{(21)}},{E_{14}^{(21)}},{E_{16}^{(21)}} \\
22 && {E_{4}^{(22)}},{E_{6}^{(22)}},{E_{8}^{(22)}},{E_{10}^{(22)}},{E_{12}^{(22)}},{E_{14}^{(22)}},{E_{16}^{(22)}},{E_{18}^{(22)}} \\
23 && {E_{4}^{(23)}},{E_{6}^{(23)}},{E_{8}^{(23)}},{E_{10}^{(23)}},{E_{12}^{(23)}} \\
26 && {E_{4}^{(26)}},{E_{6}^{(26)}},{E_{8}^{(26)}},{E_{10}^{(26)}},{E_{12}^{(26)}},{E_{14}^{(26)}} \\
29 && {E_{4}^{(29)}},{E_{6}^{(29)}},{E_{8}^{(29)}},{E_{10}^{(29)}},{E_{12}^{(29)}} \\
30 && {E_{4}^{(30)}},{E_{6}^{(30)}},{E_{8}^{(30)}},{E_{10}^{(30)}},{E_{12}^{(30)}},{E_{14}^{(30)}},{E_{16}^{(30)}},{E_{18}^{(30)}} \\
31 && {E_{4}^{(31)}},{E_{6}^{(31)}},{E_{8}^{(31)}},{E_{10}^{(31)}},{E_{12}^{(31)}} \\
33 && {E_{4}^{(33)}},{E_{6}^{(33)}},{E_{8}^{(33)}},{E_{10}^{(33)}},{E_{12}^{(33)}},{E_{14}^{(33)}} \\
34 && {E_{4}^{(34)}},{E_{6}^{(34)}},{E_{8}^{(34)}},{E_{10}^{(34)}},{E_{12}^{(34)}},{E_{14}^{(34)}},{E_{16}^{(34)}} \\
35 && {E_{4}^{(35)}},{E_{6}^{(35)}},{E_{8}^{(35)}},{E_{10}^{(35)}},{E_{12}^{(35)}},{E_{14}^{(35)}},{E_{16}^{(35)}},{E_{18}^{(35)}} \\
38 && {E_{4}^{(38)}},{E_{6}^{(38)}},{E_{8}^{(38)}},{E_{10}^{(38)}},{E_{12}^{(38)}},{E_{14}^{(38)}} \\
39 && {E_{4}^{(39)}},{E_{6}^{(39)}},{E_{8}^{(39)}},{E_{10}^{(39)}},{E_{12}^{(39)}},{E_{14}^{(39)}} \\
41 && {E_{4}^{(41)}},{E_{6}^{(41)}},{E_{8}^{(41)}},{E_{10}^{(41)}},{E_{12}^{(41)}} \\
42 && {E_{4}^{(42)}},{E_{6}^{(42)}},{E_{8}^{(42)}},{E_{10}^{(42)}},{E_{12}^{(42)}},{E_{14}^{(42)}},{E_{16}^{(42)}},{E_{18}^{(42)}} \\
46 && {E_{4}^{(46)}},{E_{6}^{(46)}},{E_{8}^{(46)}},{E_{10}^{(46)}},{E_{12}^{(46)}} \\
47 && {E_{4}^{(47)}},{E_{6}^{(47)}},{E_{8}^{(47)}},{E_{10}^{(47)}},{E_{12}^{(47)}} \\
51 && {E_{4}^{(51)}},{E_{6}^{(51)}},{E_{8}^{(51)}},{E_{10}^{(51)}},{E_{12}^{(51)}},{E_{14}^{(51)}} \\
55 && {E_{4}^{(55)}},{E_{6}^{(55)}},{E_{8}^{(55)}},{E_{10}^{(55)}},{E_{12}^{(55)}},{E_{14}^{(55)}},{E_{16}^{(55)}},{E_{18}^{(55)}},{E_{20}^{(55)}},{E_{22}^{(55)}} \\
59 && {E_{4}^{(59)}},{E_{6}^{(59)}},{E_{8}^{(59)}},{E_{10}^{(59)}},{E_{12}^{(59)}} \\
62 && {E_{4}^{(62)}},{E_{6}^{(62)}},{E_{8}^{(62)}},{E_{10}^{(62)}},{E_{12}^{(62)}} \\
66 && {E_{4}^{(66)}},{E_{6}^{(66)}},{E_{8}^{(66)}},{E_{10}^{(66)}},{E_{12}^{(66)}},{E_{14}^{(66)}},{E_{16}^{(66)}},{E_{18}^{(66)}},{E_{20}^{(66)}},{E_{22}^{(66)}} \\
69 && {E_{4}^{(69)}},{E_{6}^{(69)}},{E_{8}^{(69)}},{E_{10}^{(69)}},{E_{12}^{(69)}} \\
70 && {E_{4}^{(70)}},{E_{6}^{(70)}},{E_{8}^{(70)}},{E_{10}^{(70)}},{E_{12}^{(70)}},{E_{14}^{(70)}},{E_{16}^{(70)}},{E_{18}^{(70)}},{E_{20}^{(70)}},{E_{22}^{(70)}} \\
71 && {E_{4}^{(71)}},{E_{6}^{(71)}},{E_{8}^{(71)}},{E_{10}^{(71)}},{E_{12}^{(71)}} \\
78 && {E_{4}^{(78)}},{E_{6}^{(78)}},{E_{8}^{(78)}},{E_{10}^{(78)}},{E_{12}^{(78)}},{E_{14}^{(78)}},{E_{16}^{(78)}},{E_{18}^{(78)}} \\
87 && {E_{4}^{(87)}},{E_{6}^{(87)}},{E_{8}^{(87)}},{E_{10}^{(87)}},{E_{12}^{(87)}} \\
94 && {E_{4}^{(94)}},{E_{6}^{(94)}},{E_{8}^{(94)}},{E_{10}^{(94)}},{E_{12}^{(94)}} \\
95 && {E_{4}^{(95)}},{E_{6}^{(95)}},{E_{8}^{(95)}},{E_{10}^{(95)}},{E_{12}^{(95)}},{E_{14}^{(95)}},{E_{16}^{(95)}} \\
105 && {E_{4}^{(105)}},{E_{6}^{(105)}},{E_{8}^{(105)}},{E_{10}^{(105)}},{E_{12}^{(105)}},{E_{14}^{(105)}},{E_{16}^{(105)}},{E_{18}^{(105)}},{E_{20}^{(105)}} \\
110 && {E_{4}^{(110)}},{E_{6}^{(110)}},{E_{8}^{(110)}},{E_{10}^{(110)}},{E_{12}^{(110)}},{E_{14}^{(110)}},{E_{16}^{(110)}},{E_{18}^{(110)}},{E_{20}^{(110)}},{E_{22}^{(110)}},{E_{24}^{(110)}},{E_{26}^{(110)}} \\
119 && {E_{4}^{(119)}},{E_{6}^{(119)}},{E_{8}^{(119)}},{E_{10}^{(119)}},{E_{12}^{(119)}},{E_{14}^{(119)}},{E_{16}^{(119)}},{E_{18}^{(119)}},{E_{20}^{(119)}},{E_{22}^{(119)}},{E_{24}^{(119)}} \\
\end{array}
$$
\normalsize
\vspace*{4cm} %
\end{table}

\begin{remark}\rm
Expressions that are based on the track record of the linear algebra depend on how the base change
is made through Gauss elimination. In particular, there may be linearly dependent functions, some
of which survive the Gauss elimination while others get annihilated. We sought to express our
results in terms of Eisenstein series whose weights are as small as possible. In other words, in
the Gauss elimination we prioritized the holomorphic modular functions accordingly.

By expressing the Hauptmodul in terms of holomorphic Eisenstein series of smallest possible
weights, we were able to determine a finite list of holomorphic Eisenstein series which
generates the rational function field. Let $G$ denote any modular form of weight $k$ and
consider the function
$$
F(z)=G(z)\left(E_6^{(N)}(z)\right)^{n_6}\left(E_4^{(N)}(z)\right)^{n_4}
\Big/\big(\Delta_N(z)\big)^{nM_N},
$$
with non-negative integers $n_6$, $n_4$, and $n$ such that $k+6n_6+4n_4=k_NnM_N$. There is a
rational function $R$ in one variable such that $F(z) = R(j_{N}(z))$. Therefore, we conclude
that $G$ can be written as a rational function in terms of the holomorphic Eisenstein series that
are listed in Table~\ref{E generators}.
\end{remark}

\begin{remark}\rm
We note that the sets in Table~\ref{E generators} are not necessarily minimal sets of generators.
A specific example in the case $N=2$ is discussed below. As stated in the introduction, our
goal was to determine a set of generators of the function field. Indeed, it seems to be a
difficult problem to determine the structure of the ring of modular forms in any setting when
$M_N> 1$, meaning when there is an expression for the $M_N$-th power of the Kronecker limit
function in terms of holomorphic Eisenstein series yet no apparent expression for any smaller
power of the Kronecker limit function.
\end{remark}

\section{Examples} \label{sec:examples}

In this section we will present a number of specific formulae for various levels. It seems as
if each level has its own idiosyncratic characteristics, so we choose various examples which,
in our opinion, depict some of the most comprehensible and quantifiable nuances.

\subsection{$\mathbf{N=2}$}

We will cite specific results here, referring the reader to the article \cite{MNS} for
additional information and proofs. The Kronecker limit function can be written as
\begin{align}\label{Delta_2 formula}
\Delta_2(z) = \tfrac{ 17 }{ 1152 } \big(E^{(2)}_{4}(z)\big)^2 - \tfrac{ 17 }{ 1152 } E^{(2)}_{8}(z).
\end{align}
In addition, one has that
$$
j_2(z)\Delta_2(z) = - \tfrac{ 77 }{ 144 } \big(E^{(2)}_{4}(z)\big)^2 + \tfrac{ 221 }{ 144 } E^{(2)}_{8}(z).
$$
By arguing as in \cite{Serre73}, one can prove a dimension formula for the space of automorphic
forms of weight $2k$, namely that
\begin{align}\label{dimension f-la N=2}
\dim \T_{2k}^{(2)} = \begin{cases}
\lfloor \frac{k}{4} \rfloor, & \text{if $k$ is congruent to $1$ modulo $4$, $k\geq 0$} \\
\lfloor \frac{k}{4} \rfloor +1, & \text{if $k$ is not congruent to $1$ modulo $4$, $k\geq 0$}.
\end{cases}
\end{align}
The space $\T_{2k}^{(2)}$ is generated by the set of monomials
$(E_4^{(2)}(z))^{l}(E_6^{(2)}(z))^{m}(E_8^{(2)}(z))^{n} $, where $l,m,n$ are non-negative integers
such that $4l + 6m + 8n =2k$. The dimension formula~\eqref{dimension f-la N=2} yields some
interesting number-theoretical formulae. For example, since $\dim \T_{10}^{(2)}=1$, we see
that $E_{10}^{(2)}(z)=E_{6}^{(2)}(z)E_{4}^{(2)}(z)$. By equating the
$q$-expansions~\eqref{E_k^N q-expansion} for $k \in \{2,3,5\}$, one obtains the following
summation formula for the generalized sum of divisors:
$$
A_{9}^{(2)}(n) = 336 \sum_{j=1}^{n-1} A_3^{(2)}(j)A_5^{(2)}(n-j) + 7 A_5^{(2)}(n) - 6 A_3^{(2)}(n),
$$
where $A_{2k-1}^{(2)}(n) = \sigma_{2k-1}(n) + 2^k \delta(n)\sigma_{2k-1} (n/2)$, for $k=1,2,\ldots$
and $\delta(n)= 1$ for even positive integers $n$ and $\delta(n) = 0$, otherwise.

Analogously, using formula~\eqref{Delta_2 formula}, the $q$-expansion~\eqref{E_k^N q-expansion}
and the $q$-expansion for the delta function, $\Delta(z) = \sum_{n=1}^{\infty} \tau(n) q^n$,
where $\tau(n)$ is the Ramanujan function, one obtains relations involving $\tau$, $\sigma_3$
and $\sigma_7$.

\subsection{$\mathbf{N=3}$}

As with the case $N=2$, we refer the reader to \cite{MNS} for additional information and proofs.
The Kronecker limit function vanishes to order $2$ at $i\infty$ and has weight $12$. The
smallest weight cusp form has weight $8$, but it vanishes to order $1$ at $i\infty$, and,
consequently, it vanishes elsewhere. The Kronecker limit function can be written as
\begin{equation}\label{jd_n3}
j_3(z)\Delta_3(z) = \tfrac{ 541 }{ 1728 } \big(E^{(3)}_{4}(z)\big)^3 + \tfrac{ 14461 }{ 24300 } \big(E^{(3)}_{6}(z)\big)^2 - \tfrac{ 353101 }{ 388800 } E^{(3)}_{12}(z)
\end{equation}
and the Hauptmodul is given by
\begin{equation}\label{d_n3}
\Delta_3(z) = - \tfrac{ 25 }{ 3456 } \big(E^{(3)}_{4}(z)\big)^3 - \tfrac{ 1049 }{ 72900 } \big(E^{(3)}_{6}(z)\big)^2 + \tfrac{ 50443 }{ 2332800 } E^{(3)}_{12}(z).
\end{equation}
The dimension formula for the space of automorphic forms of weight $2k$ is
$$
\dim \T_{2k}^{(3)} = \begin{cases}
\lfloor \frac{k}{3} \rfloor, & \text{if $k$ is congruent to $1$ or $3$ modulo $6$, $k\geq 0$} \\
\lfloor \frac{k}{3} \rfloor +1, & \text{if $k$ is not congruent to $1$ or $3$ modulo $6$, $k\geq 0$}.
\end{cases}
$$

We note that the forms $E_{8}^{(3)}(z) - (E_{4}^{(3)}(z))^{2}$ and
$E_{10}^{(3)}(z) - E_{4}^{(3)}(z)E_{6}^{(3)}(z)$ are cusp forms which vanish
at elliptic points on $X_{3}$; see Appendix B of \cite{MNS}. In other words,
there are cusp forms of weight smaller than the weight of
the Kronecker limit function, but these forms necessarily vanish at some point
in the interior of $X_{3}$, whereas the Kronecker limit function
vanishes at $i\infty$ only.

Finally, let us explain why $E_{8}^{(3)}$ does not appear in Table 2.  The information
in Appendix B of \cite{MNS} describes the zeros of small weight holomorphic forms.
In particular, we conclude from the information provided that
$$
\frac{E_{8}^{(3)}(z)}{E_{8}^{(3)}(z) - (E_{4}^{(3)}(z))^{2}} = c_{1}j_{3}(z) + c_{2}
$$
for some explicitly computable constants $c_{1}$ and $c_{2}$.  From this, we get that
\begin{equation}\label{e8_n3}
E_{8}^{(3)}(z) = (E_{4}^{(3)}(z))^{2} \frac{c_{1}j_{3}(z) + c_{2}}{c_{1}j_{3}(z) + c_{2}-1}.
\end{equation}
When combining (\ref{jd_n3}), (\ref{d_n3}) and (\ref{e8_n3}), we get a formula which expresses
$E_{8}^{(3)}$ as a rational function involving $E_{4}^{(3)}$, $E_{6}^{(3)}$ and $E_{12}^{(3)}$, as
asserted by Table 2.

\subsection{$\mathbf{N=5}$}

In the case $N=5$, the surface $X_5$ has genus zero, three order two elliptic elements
$e_1=i/\sqrt{5}$, $e_2=2/5 + i/5$, $e_3=1/2 + i/(2\sqrt{5})$, and one cusp, hence
$\vol_{\text{hyp}}(X_5)=\pi$. Its Kronecker limit
function has weight four, which is minimal, and the function vanishes at $i\infty$ to
order one, which is also minimal.
As a result, we have that the mapping $f \mapsto \Delta_5 f$ is an isometry between
the spaces $\T_{2k-4}^{(5)}$ and $S_{2k}^{(5)}$; therefore, we arrive at the dimension formula
$$
\dim \T_{2k}^{(2)} = \begin{cases}
\lfloor \frac{k}{2} \rfloor, & \text{if $k$ is congruent to $1$ modulo $2$, $k\geq 0$} \\
\lfloor \frac{k}{2} \rfloor +1, & \text{if $k$ is not congruent to $0$ modulo $2$, $k\geq 0$}.
\end{cases}
$$
The space $\T_{2k}^{(5)}$ is generated by the set of monomials
$(E_4^{(5)}(z))^l (\Delta_5(z))^m (E_6^{(5)}(z))^n$, where $l,m,n$ are non-negative integers
such that $4l+4m+6n=2k$. From the output of the \jst2, we have that
$$
j_5(z)\Delta_5(z) = E^{(5)}_{4}(z) - \tfrac{ 172 }{ 13 } \Delta_5(z).
$$
The analysis of $\Delta_{5}^{3}$ differs between the {\bf JST2} and {\bf JST3} algorithms.
From {\bf JST2}, we have that $\Delta_{5}^{3}$ is a rational function in the holomorphic
Eisenstein series of weights four, six, eight and twelve. From {\bf JST3}, we have that
$\Delta_{5}^{3}$ is a \textit{polynomial} in the holomorphic Eisenstein series of weights
four, six, eight and twelve. Namely, from the output of the \jst3, we have that
$$
j_5(z)\big(\Delta_5(z)\big)^3 = \tfrac{ 10330419229 }{ 11016000000 } \big(E^{(5)}_{4}(z)\big)^3 + \tfrac{ 36659 }{ 2448000 } \big(E^{(5)}_{6}(z)\big)^2 - \tfrac{ 28493266087 }{ 11016000000 } E^{(5)}_{8}(z) E^{(5)}_{4}(z) + \tfrac{ 2999646893 }{ 1836000000 } E^{(5)}_{12}(z)
$$
and
$$
\big(\Delta_5(z)\big)^3 = - \tfrac{ 9383387 }{ 162000000 } \big(E^{(5)}_{4}(z)\big)^3 - \tfrac{ 13 }{ 9000 } \big(E^{(5)}_{6}(z)\big)^2 + \tfrac{ 3226717 }{ 20250000 } E^{(5)}_{8}(z) E^{(5)}_{4}(z) - \tfrac{ 5398783 }{ 54000000 } E^{(5)}_{12}(z).
$$

\subsection{$\mathbf{N=6}$}

Topologically, $X_{5}$ and $X_{6}$ are identical, with the same number of cusps, elliptic
points of order two, and consequently, the same hyperbolic volume. The {\bf JST2} and
{\bf JST3} algorithms performed similarly in both cases, as one can see from
Table~\ref{tab:jst_2_vs_3} and Table~\ref{E generators}.
All comments above regarding the holomorphic function theory for
$X_{5}$ hold for $X_{6}$. However, as show in \cite{JST12}, the
analytic function theory of $X_{5}$ and $X_{6}$ are different. Specifically, the counting
functions for the analytic Maass forms, when ordered by their
Laplacian eigenvalues, are shown to be equal in their lead term but unequal in lower order terms.

\subsection{$\mathbf{N=17}$}

As we stated in the introduction, as $N$ becomes larger, the formulae become massive.
Our last example for $N=17$.
The Kronecker limit function has weight four and vanishes at $i\infty$ to order four.
From the \jst3, we have the following formulae:
\begin{align*}
j_{17}&(z)\big(\Delta_{17}(z)\big)^9 = \tfrac{ 81682801889356820001790224970058471917613108127362192461613220533 }{ 3269846855773492420944242299705431901325975126578932604974661632 } \big(E^{(17)}_{4}(z)\big)^9\\&
- \tfrac{ 57998022455299820152689336251300228068357304045275286805301 }{ 1197457521190948626327504142387996791894290229520024731648 } \big(E^{(17)}_{6}(z)\big)^2 \big(E^{(17)}_{4}(z)\big)^6\\&
+ \tfrac{ 40497436515338798408532045523225489025965457561330556316291 }{ 1852774239194857326466652706713276353684752025138495488000 } \big(E^{(17)}_{6}(z)\big)^4 \big(E^{(17)}_{4}(z)\big)^3\\&
- \tfrac{ 3758480257690225061233693208729793594924453574315163 }{ 550341367907988517797569501748755788899740024832000 } \big(E^{(17)}_{6}(z)\big)^6\\&
- \tfrac{ 19414695740146736017085565287911573267947533788487530546931336997391 }{ 235020242758719767755367415291327917907804462222860780982553804800 } E^{(17)}_{8}(z) \big(E^{(17)}_{4}(z)\big)^7\\&
+ \tfrac{ 87392429573930662513617849766871793131053840056808626436739811257 }{ 429213680251880648249264766037197600094609654143583864750080000 } E^{(17)}_{8}(z) \big(E^{(17)}_{6}(z)\big)^2 \big(E^{(17)}_{4}(z)\big)^4\\&
- \tfrac{ 5203291809002722923420727059042670529678338299681100572348497 }{ 159222786180808051493227966983172186644783377160339456000000 } E^{(17)}_{8}(z) \big(E^{(17)}_{6}(z)\big)^4 E^{(17)}_{4}(z)\\&
- \tfrac{ 3408021881707620602850044141317857445104537752516243513916285865231 }{ 546558704090045971524110268119367250948382470285722746471055360000 } \big(E^{(17)}_{8}(z)\big)^2 \big(E^{(17)}_{4}(z)\big)^5\\&
- \tfrac{ 16613503534705813629198888518084937494696284987808069450102921171 }{ 95380817833751255166503281341599466687691034254129747722240000 } \big(E^{(17)}_{8}(z)\big)^2 \big(E^{(17)}_{6}(z)\big)^2 \big(E^{(17)}_{4}(z)\big)^2\\&
+ \tfrac{ 2084310764069464266375452181379302123943671896630614730410490282708941 }{ 24481275287366642474517439092846658115396298148214664685682688000000 } \big(E^{(17)}_{8}(z)\big)^3 \big(E^{(17)}_{4}(z)\big)^3\\&
- \tfrac{ 9922136522478992021059089148544040599546174236808342061149389 }{ 600269903901646354129469435526559143650833331894479749120000 } \big(E^{(17)}_{8}(z)\big)^3 \big(E^{(17)}_{6}(z)\big)^2\\&
+ \tfrac{ 39971724261482388723963548784518805970985444209456555554807081177551 }{ 580296895700542636433005963682291155327912252402125385142108160000 } \big(E^{(17)}_{8}(z)\big)^4 E^{(17)}_{4}(z)\\&
- \tfrac{ 12179813594881425731721530954876395006064827564865712237231709007 }{ 111595556865488968544808839169671376024598510077331804835020800 } E^{(17)}_{10}(z) E^{(17)}_{6}(z) \big(E^{(17)}_{4}(z)\big)^5\\&
- \tfrac{ 1625258630148098158844608861059428762679867410760523654188493 }{ 305707749467151458866997696607690598357984084147851755520000 } E^{(17)}_{10}(z) \big(E^{(17)}_{6}(z)\big)^3 \big(E^{(17)}_{4}(z)\big)^2\\&
+ \tfrac{ 201956165169824936446796453214912198922237546648448042931247350643 }{ 697472230409306053405055244810446100153740687983323780218880000 } E^{(17)}_{10}(z) E^{(17)}_{8}(z) E^{(17)}_{6}(z) \big(E^{(17)}_{4}(z)\big)^3\\&
+ \tfrac{ 2813492092804509777019550484166164259166371672470004179865987 }{ 110819059181842403839286665020287841904769230503596261376000 } E^{(17)}_{10}(z) E^{(17)}_{8}(z) \big(E^{(17)}_{6}(z)\big)^3\\&
- \tfrac{ 298150566267220507427939657298379519874877536521040468383984890437 }{ 2789888921637224213620220979241784400614962751933295120875520000 } E^{(17)}_{10}(z) \big(E^{(17)}_{8}(z)\big)^2 E^{(17)}_{6}(z) E^{(17)}_{4}(z)\\&
+ \tfrac{ 270142921637107712433606444209937936587102913556256973385136064101 }{ 1339146682385867622537706070036056512295182120927981658020249600 } \big(E^{(17)}_{10}(z)\big)^2 \big(E^{(17)}_{4}(z)\big)^4\\&
- \tfrac{ 118581244701243654625492701946386671875131600350921745366238181 }{ 1324733581024322988423656685299992592884597697974024273920000 } \big(E^{(17)}_{10}(z)\big)^2 \big(E^{(17)}_{6}(z)\big)^2 E^{(17)}_{4}(z)\\&
- \tfrac{ 622123658423176240663440801764445330920253599234807877389715511211 }{ 1287641040755641944747794298111592800283828962430751594250240000 } \big(E^{(17)}_{10}(z)\big)^2 E^{(17)}_{8}(z) \big(E^{(17)}_{4}(z)\big)^2\\&
- \tfrac{ 41865928593946018666000515901728601665582738769261235242020892533923 }{ 970881344729754026339836900776140971414007037672786702064680960000 } \big(E^{(17)}_{10}(z)\big)^2 \big(E^{(17)}_{8}(z)\big)^2\\&
+ \tfrac{ 226293023118065631604956427915797705760362816981994085929537 }{ 2676229456614793916007387243030288066433530702977826816000 } \big(E^{(17)}_{10}(z)\big)^3 E^{(17)}_{6}(z)\\&
+ \tfrac{ 65939988441096018663885334259995469677 }{ 1934740380512325216435446698160947200 } E^{(17)}_{12}(z) \big(E^{(17)}_{4}(z)\big)^6\\&
- \tfrac{ 47508694350054116027131578232203571279280588142292406574968917 }{ 496775092884121120658871256987497222331724136740259102720000 } E^{(17)}_{12}(z) \big(E^{(17)}_{6}(z)\big)^2 \big(E^{(17)}_{4}(z)\big)^3\\&
+ \tfrac{ 83239130800439493554048989758304142656207424379566457431 }{ 2818127349459872237687554172747939125917634265088000000 } E^{(17)}_{12}(z) \big(E^{(17)}_{6}(z)\big)^4\\&
+ \tfrac{ 18888098569562683617871650219704377045625972851708713833077441 }{ 99355018576824224131774251397499444466344827348051820544000 } E^{(17)}_{12}(z) E^{(17)}_{10}(z) E^{(17)}_{6}(z) \big(E^{(17)}_{4}(z)\big)^2\\&
- \tfrac{ 6458100747185096513157918629463271052359 }{ 48368509512808130410886167454023680000 } \big(E^{(17)}_{12}(z)\big)^2 \big(E^{(17)}_{4}(z)\big)^3\\&
+ \tfrac{ 13456181814083822984196529705199819074619 }{ 77183791775757654910988565086208000000 } \big(E^{(17)}_{12}(z)\big)^3
\end{align*}
and
\begin{align*}
\big(\Delta_{17}&(z)\big)^9 = - \tfrac{ 4410175152266863630497017095287573799108101287320169 }{ 513269785149806673002504728644869309011527884341248 } \big(E^{(17)}_{4}(z)\big)^9\\&
+ \tfrac{ 19865215328281078919219868581830673116116281279861 }{ 1077982783479943155001973186393454060492470353920 } \big(E^{(17)}_{6}(z)\big)^2 \big(E^{(17)}_{4}(z)\big)^6\\&
- \tfrac{ 1147994099850642662275857201554136108251932243 }{ 116332425049635581779544719243012827040972800 } \big(E^{(17)}_{6}(z)\big)^4 \big(E^{(17)}_{4}(z)\big)^3\\&
+ \tfrac{ 257163348099153057405937570593213576401 }{ 86387408378599422207040509547266048000 } \big(E^{(17)}_{6}(z)\big)^6\\&
+ \tfrac{ 4110275602561195487616512760454051197916582070385787933 }{ 147565063230569418488220109485399926340814266748108800 } E^{(17)}_{8}(z) \big(E^{(17)}_{4}(z)\big)^7\\&
- \tfrac{ 2455783752311086170178917777522781426586892700694851 }{ 33686961983748223593811662074795439390389698560000 } E^{(17)}_{8}(z) \big(E^{(17)}_{6}(z)\big)^2 \big(E^{(17)}_{4}(z)\big)^4\\&
+ \tfrac{ 1581775255838347728745765778157179068844765441801 }{ 99973177777030578091796243099464148238336000000 } E^{(17)}_{8}(z) \big(E^{(17)}_{6}(z)\big)^4 E^{(17)}_{4}(z)\\&
+ \tfrac{ 6518162027197225998646914560331274207300504121132929847 }{ 1844563290382117731102751368567499079260178334351360000 } \big(E^{(17)}_{8}(z)\big)^2 \big(E^{(17)}_{4}(z)\big)^5\\&
+ \tfrac{ 877428475040946870505912480572673877165899233742103 }{ 14971983103888099375027405366575750840173199360000 } \big(E^{(17)}_{8}(z)\big)^2 \big(E^{(17)}_{6}(z)\big)^2 \big(E^{(17)}_{4}(z)\big)^2\\&
- \tfrac{ 905386954382815429576749294608296584568282296576830392199 }{ 30742721506368628851712522809458317987669638905856000000 } \big(E^{(17)}_{8}(z)\big)^3 \big(E^{(17)}_{4}(z)\big)^3\\&
+ \tfrac{ 207452833460189538372130707619778910743078771025457 }{ 36182292501062906822982896302558064530418565120000 } \big(E^{(17)}_{8}(z)\big)^3 \big(E^{(17)}_{6}(z)\big)^2\\&
- \tfrac{ 6765708051219903828398888390867858547193923080811280631 }{ 273268635612165589793000202750740604334841234718720000 } \big(E^{(17)}_{8}(z)\big)^4 E^{(17)}_{4}(z)\\&
+ \tfrac{ 7025876804004356240055790621469114807967838822691 }{ 194096623064255692728887138824306132775652556800 } E^{(17)}_{10}(z) E^{(17)}_{6}(z) \big(E^{(17)}_{4}(z)\big)^5\\&
+ \tfrac{ 203944326653207551761076174261325691779672537 }{ 265856650044181038692865355610763385896960000 } E^{(17)}_{10}(z) \big(E^{(17)}_{6}(z)\big)^3 \big(E^{(17)}_{4}(z)\big)^2\\&
- \tfrac{ 88050066607840362983543089832425254378757106875786733 }{ 875861011577453813439103213944681424150132162560000 } E^{(17)}_{10}(z) E^{(17)}_{8}(z) E^{(17)}_{6}(z) \big(E^{(17)}_{4}(z)\big)^3\\&
- \tfrac{ 114754891200905341203611097297729345611512888589 }{ 13916266346562656470378037039445409434776371200 } E^{(17)}_{10}(z) E^{(17)}_{8}(z) \big(E^{(17)}_{6}(z)\big)^3\\&
+ \tfrac{ 66394449571915938902069992307694226884439057160324951 }{ 1751722023154907626878206427889362848300264325120000 } E^{(17)}_{10}(z) \big(E^{(17)}_{8}(z)\big)^2 E^{(17)}_{6}(z) E^{(17)}_{4}(z)\\&
- \tfrac{ 3664823227867792880990102284616506270153441589203601 }{ 52551660694647228806346192836680885449007929753600 } \big(E^{(17)}_{10}(z)\big)^2 \big(E^{(17)}_{4}(z)\big)^4\\&
+ \tfrac{ 14427551517079169370308214911137713786234600256489 }{ 415888419552447204861872371293770856671477760000 } \big(E^{(17)}_{10}(z)\big)^2 \big(E^{(17)}_{6}(z)\big)^2 E^{(17)}_{4}(z)\\&
+ \tfrac{ 34114607946828890598140117698005174033430842408608771 }{ 202121771902489341562869972448772636342338191360000 } \big(E^{(17)}_{10}(z)\big)^2 E^{(17)}_{8}(z) \big(E^{(17)}_{4}(z)\big)^2\\&
+ \tfrac{ 1186277940138861135501685541633367245343342399481395343 }{ 76199908007238481769201979613187283901061498142720000 } \big(E^{(17)}_{10}(z)\big)^2 \big(E^{(17)}_{8}(z)\big)^2\\&
- \tfrac{ 2240074672005345691936094582673021223667749558747 }{ 73095540406187690551480598591026392990744576000 } \big(E^{(17)}_{10}(z)\big)^3 E^{(17)}_{6}(z)\\&
- \tfrac{ 473713406463236803998887 }{ 40792493008974879129600 } E^{(17)}_{12}(z) \big(E^{(17)}_{4}(z)\big)^6\\&
+ \tfrac{ 187627181944563944553278965376532704223825410987 }{ 5030908301037667800748456104360131330703360000 } E^{(17)}_{12}(z) \big(E^{(17)}_{6}(z)\big)^2 \big(E^{(17)}_{4}(z)\big)^3\\&
- \tfrac{ 3176432730003963047610699437833910664552631 }{ 221181554210724382719750270133948416000000 } E^{(17)}_{12}(z) \big(E^{(17)}_{6}(z)\big)^4\\&
- \tfrac{ 2144787933823513840784295072436611609578065848101 }{ 31191631466433540364640427847032814250360832000 } E^{(17)}_{12}(z) E^{(17)}_{10}(z) E^{(17)}_{6}(z) \big(E^{(17)}_{4}(z)\big)^2\\&
+ \tfrac{ 46882982116711758510391631 }{ 1019812325224371978240000 } \big(E^{(17)}_{12}(z)\big)^2 \big(E^{(17)}_{4}(z)\big)^3\\&
- \tfrac{ 4639965815125172171338200503 }{ 76485924391827898368000000 } \big(E^{(17)}_{12}(z)\big)^3.
\end{align*}

Using the exact identity of the Hauptmodul in terms of Eisenstein series, we can read off that
${E_{4}^{(17)}}(z)$, ${E_{6}^{(17)}}(z)$, ${E_{8}^{(17)}}(z)$, ${E_{10}^{(17)}}(z)$, and
${E_{12}^{(17)}}(z)$ generate the holomorphic Eisenstein series $E_{k}^{(17)}(z)$
for all even $k\ge4$.

\section{Concluding remarks}

\subsection{Known relations for Hauptmoduli}

\begin{table}
\caption{\label{tab:j known}Known expressions of the Hauptmoduli $j_N$ for the genus zero groups $\Gamma_0(N)^+$.}
\small
$$
\begin{array}{rcl}
N && \qquad \text{Formula for } t_N =j_N + const. \\[1ex]
2 && t_2=\big(\frac{\eta(z)}{\eta(2z)} \big)^{24} + 4096 \big(\frac{\eta(2z)}{\eta(z)} \big)^{24} \\
3 && t_3=\big(\frac{\eta(z)}{\eta(3z)} \big)^{12} + 729 \big(\frac{\eta(3z)}{\eta(z)} \big)^{12} \\
5 && t_5=\big(\frac{\eta(z)}{\eta(5z)} \big)^{6} + 125 \big(\frac{\eta(5z)}{\eta(z)} \big)^{6} \\
6 && t_6 = \big(\frac{\eta(z) \eta(2z)}{\eta(3z) \eta(6z)} \big)^4 + 81 \big(\frac{\eta(3z) \eta(6z)}{\eta(z) \eta(2z)} \big)^4 =\big(\frac{\eta(z) \eta(3z)}{\eta(2z) \eta(6z)} \big)^6 + 64 \big(\frac{\eta(2z) \eta(6z)}{\eta(z) \eta(3z)} \big)^6 +c_1 = \big(\frac{\eta(2z) \eta(3z)}{\eta(z) \eta(6z)} \big)^{12} + \big(\frac{\eta(z) \eta(6z)}{\eta(2z) \eta(3z)} \big)^{12} +c_2 \\
7 && t_7=\big(\frac{\eta(z)}{\eta(7z)} \big)^{4} + 49 \big(\frac{\eta(7z)}{\eta(z)} \big)^{4} \\
10 && t_{10} = \big(\frac{\eta(z) \eta(2z)}{\eta(5z) \eta(10z)} \big)^2 + 25 \big(\frac{\eta(5z) \eta(10z)}{\eta(z) \eta(2z)} \big)^2 = \big(\frac{\eta(z) \eta(5z)}{\eta(2z) \eta(10z)} \big)^4 + 16 \big(\frac{\eta(2z) \eta(10z)}{\eta(z) \eta(5z)} \big)^4 +c_1 = \big(\frac{\eta(2z) \eta(5z)}{\eta(z) \eta(10z)} \big)^6 + \big(\frac{\eta(z) \eta(10z)}{\eta(2z) \eta(5z)} \big)^6 + c_2\\
11 && t_{11}= \big(\frac{\theta(2,2,6)}{\eta(z)\eta(11z)} \big)^2 = \big(\frac{\eta(z) \eta(11z)}{\eta(2z) \eta(22z)} \big)^2 + 16 \big(\frac{\eta(2z) \eta(22z)}{\eta(z) \eta(11z)} \big)^2 +16 \big(\frac{\eta(2z) \eta(22z)}{\eta(z) \eta(11z)} \big)^4 +c_1 \\
13 && t_{13} = \big(\frac{\eta(z)}{\eta(13z)} \big)^{2} + 13 \big(\frac{\eta(13z)}{\eta(z)} \big)^{2} \\
14 && t_{14}= \big(\frac{\eta(z) \eta(7z)}{\eta(2z) \eta(14z)} \big)^3 + 8\big(\frac{\eta(2z) \eta(14z)}{\eta(z) \eta(7z)} \big)^3 = \big(\frac{\eta(2z) \eta(7z)}{\eta(z) \eta(14z)} \big)^4 + \big(\frac{\eta(z) \eta(14z)}{\eta(2z) \eta(7z)} \big)^4 +c_1 \\
15 && t_{15}= \big(\frac{\eta(z) \eta(5z)}{\eta(3z) \eta(15z)} \big)^2 + 9\big(\frac{\eta(3z) \eta(15z)}{\eta(z) \eta(5z)} \big)^2 = \big(\frac{\eta(3z) \eta(5z)}{\eta(z) \eta(15z)} \big)^3 + \big(\frac{\eta(z) \eta(15z)}{\eta(3z) \eta(5z)} \big)^3 +c_1\\
17 && t_{17}= \big( \frac{\theta_x(\frac{1}{2}, 0, \frac{17}{2}) -\theta_y(\frac{1}{2}, 0, \frac{17}{2}) }{2\eta(z)\eta(17z)}\big)^2 \\
19 && t_{19}= \big(\frac{2\theta(2,2,10)}{\theta(1,2,20) -\theta(4,2,5)} \big)^2 \\
21 && t_{21}= \big(\frac{\eta(z) \eta(3z)}{\eta(7z) \eta(21z)} \big) + 7 \big(\frac{\eta(7z) \eta(21z)}{\eta(z) \eta(3z)} \big) = \big(\frac{\eta(3z) \eta(7z)}{\eta(z) \eta(21z)} \big)^2 + \big(\frac{\eta(z) \eta(21z)}{\eta(3z) \eta(7z)} \big)^2 +c_1 \\
22 && t_{22} = \big(\frac{\eta(z) \eta(11z)}{\eta(2z) \eta(22z)} \big)^2 + 4 \big(\frac{\eta(2z) \eta(22)}{\eta(z) \eta(11z)} \big)^2 \\
23 && t_{23} =\big(\frac{\theta(2,2,12)}{\eta(z)\eta(23z)} \big) = \big(\frac{\eta(z) \eta(23z)}{\eta(2z) \eta(46z)} \big) + 4 \big(\frac{\eta(2z) \eta(46z)}{\eta(z) \eta(23z)} \big) +4 \big(\frac{\eta(2z) \eta(46z)}{\eta(z) \eta(23z)} \big) ^2 +c_1 \\
26 && t_{26}= \big(\frac{\eta(2z) \eta(13z)}{\eta(z) \eta(26z)} \big)^2 + \big(\frac{\eta(z) \eta(26z)}{\eta(2z) \eta(13z)} \big)^2 \\
29 && t_{29} = \frac{\theta_x(\frac{1}{2}, 0, \frac{29}{2}) -\theta_y(\frac{1}{2}, 0, \frac{29}{2}) }{2\eta(z)\eta(29z)} \\
30 && t_{30} = \big(\frac{\eta(z) \eta(6z) \eta(10z)\eta(15z)}{\eta(2z) \eta(3z) \eta(5z) \eta(30z)} \big)^3 + \big(\frac{\eta(z) \eta(6z) \eta(10z)\eta(15z)}{\eta(2z) \eta(3z) \eta(5z) \eta(30z)} \big)^{-3} = \big(\frac{\eta(z) \eta(3z) \eta(5z)\eta(15z)}{\eta(2z) \eta(6z) \eta(10z) \eta(30z)} \big)+ 4\big(\frac{\eta(z) \eta(3z) \eta(5z)\eta(15z)}{\eta(2z) \eta(6z) \eta(10z) \eta(30z)} \big)^{-1} + c_1 \\
30 && t_{30} = \big(\frac{\eta(3z) \eta(5z) \eta(6z)\eta(10z)}{\eta(z) \eta(2z) \eta(15z) \eta(30z)} \big) + \big(\frac{\eta(3z) \eta(5z) \eta(6z)\eta(10z)}{\eta(z) \eta(2z) \eta(15z) \eta(30z)} \big)^{-1} +c_2 = \big(\frac{\eta(2z) \eta(3z) \eta(10z)\eta(15z)}{\eta(z) \eta(5z) \eta(6z) \eta(30z)} \big)^2+ \big(\frac{\eta(2z) \eta(3z) \eta(10z)\eta(15z)}{\eta(z) \eta(5z) \eta(6z) \eta(30z)} \big)^{-2} + c_3 \\
31 && t_{31}= \big(\frac{\theta(2,2,16) - \theta(4,2,8)}{2\eta(z)\eta(31z)} \big)^3 \\
33 && t_{33}= \big(\frac{\eta(z) \eta(11z)}{\eta(3z) \eta(33z)} \big) + 3\big(\frac{\eta(z) \eta(11z)}{\eta(3z) \eta(33z)} \big)^{-1} \\
34 && t_{34} \text{ is deduced from the formula } t_{34}^2(z) + t_{34}(z) - 6=j_{17}(z) + j_{17}(2z) \\
35 && t_{35}= \big(\frac{\eta(5z) \eta(7z)}{\eta(z) \eta(35z)} \big) -\big(\frac{\eta(5z) \eta(7z)}{\eta(z) \eta(35z)} \big)^{-1} \\
38 && t_{38} \text{ is deduced from the formula } t_{38}^2(z) + t_{38}(z) - 4=j_{19}(z) + j_{19}(2z) \\
39 && t_{39}= \big(\frac{\eta(3z) \eta(13z)}{\eta(z) \eta(39z)} \big) + \big(\frac{\eta(3z) \eta(13z)}{\eta(z) \eta(39z)} \big)^{-1} \\
41 && t_{41}= \frac{\theta_x(\frac{3}{2}, 2, \frac{15}{2}) -\theta_y(\frac{3}{2}, 2, \frac{15}{2}) }{2\eta(z)\eta(41z)} \\
42 && t_{42}= \big(\frac{\eta(z) \eta(6z) \eta(14z)\eta(21z)}{\eta(2z) \eta(3z) \eta(7z) \eta(42z)} \big)^2 + \big(\frac{\eta(z) \eta(6z) \eta(14z)\eta(21z)}{\eta(2z) \eta(3z) \eta(7z) \eta(42z)} \big)^{-2} = \big(\frac{\eta(2z) \eta(6z) \eta(7z)\eta(21z)}{\eta(z) \eta(3z) \eta(14z) \eta(42z)} \big) + \big(\frac{\eta(2z) \eta(6z) \eta(7z)\eta(21z)}{\eta(z) \eta(3z) \eta(14z) \eta(42z)} \big)^{-1}+c_1 \\
46 && t_{46}= \big(\frac{\eta(z) \eta(23z)}{\eta(2z) \eta(46z)} \big) + 2\big(\frac{\eta(z) \eta(23z)}{\eta(2z) \eta(46z)} \big)^{-1} \\
47 && t_{47}= \frac{\theta(2,2,24) - \theta(4,2,12)}{2\eta(z)\eta(47z)} \\
51 && t_{51} \text{ is deduced from the formula } t_{51}^3(z) -2 t_{51}(z) - 6=j_{17}(z) + j_{17}(3z) \\
55 && t_{55} \text{ is deduced from the formula } t_{55}^5(z)-10 t_{55}^3(z) - 5t_{55}^2(z)+16 t_{55}(z)=j_{11}(z) + j_{11}(5z) \\
59 && t_{59}= \frac{2\theta(6,2,10)}{\theta(2,2,30) - \theta(6,2,10)}\\
62 && t_{62} \text{ is deduced from the formula } t_{62}^2(z)+t_{62}(z)-2=j_{31}(z)+j_{31}(2z) \\
66 && t_{66}= \big(\frac{\eta(2z) \eta(3z) \eta(22z)\eta(33z)}{\eta(z) \eta(6z) \eta(11z) \eta(66z)} \big) + \big(\frac{\eta(2z) \eta(3z) \eta(22z)\eta(33z)}{\eta(z) \eta(6z) \eta(11z) \eta(66z)} \big)^{-1}\\
69 && t_{69} \text{ is deduced from the formula } t_{69}^3(z) -2 t_{69}(z) - 3=j_{23}(z) + j_{23}(3z) \\
70 && t_{70} =\big(\frac{\eta(z) \eta(10z) \eta(14z)\eta(35z)}{\eta(2z) \eta(5z) \eta(7z) \eta(70z)} \big) +\big(\frac{\eta(z) \eta(10z) \eta(14z)\eta(35z)}{\eta(2z) \eta(5z) \eta(7z) \eta(70z)} \big)^{-1} \\
71 && t_{71}= \frac{\theta(4,2,18) - \theta(6,2,12)}{2\eta(z)\eta(71z)} \\
78 && t_{78}= \big(\frac{\eta(z) \eta(6z) \eta(26z)\eta(39z)}{\eta(2z) \eta(3z) \eta(13z) \eta(78z)} \big) + \big(\frac{\eta(z) \eta(6z) \eta(26z)\eta(39z)}{\eta(2z) \eta(3z) \eta(13z) \eta(78z)} \big)^{-1} \\
87 && t_{87} \text{ is deduced from the formula } t_{87}^3(z)+ t_{87}(z)-3 =j_{29}(z) + j_{29}(3z) \\
94 && t_{94} \text{ is deduced from the formula } t_{94}^2(z)+ t_{94}(z) -2 =j_{47}(z) + j_{47}(2z) \\
95 && t_{95} \text{ is deduced from the formula } t_{95}^5(z)-3 t_{95}^3(z)+ t_{95}(z) -2=j_{19}(z) + j_{19}(5z) \\
105 && t_{105} \text{ is deduced from the formula } t_{105}^3(z)-2 t_{105}(z) -3 =j_{35}(z) + j_{35}(3z) \\
110 && t_{110} \text{ is deduced from the formula } t_{110}^2(z)+ t_{110}(z) =j_{55}(z) + j_{55}(2z) \\
119 && t_{119} \text{ is deduced from the formula } t_{119}^7(z) - 7t_{119}^3(z) -7 t_{119}^2(z) -6 t_{119}(z) -7 =j_{17}(z) + j_{17}(7z) \\
\end{array}
$$
\normalsize
\vspace*{4cm} %
\end{table}

In \cite{CN79}, the authors computed expressions for $j_{N}$, up to an additive
constant. The authors call their function $t_{N}$. The data from \cite{CN79}
relates to genus zero groups $\Gamma_0(N)^+$ with square-free level $N$
as given in Table~\ref{tab:j known}, using the Dedekind eta function together with
$\theta(a,b,c)$, which is the theta function defined by the series
$$
\theta(a,b,c)= \sum_{(x,y)\in \Z^2} q^{\frac{1}{2}(ax^2 + bxy + cy^2)}.
$$
Additionally, one has, in the notation of \cite{CN79}, the functions $\theta_x(a,b,c)$
and $\theta_y(a,b,c)$ which are defined by the same series which defines $\theta(a,b,c)$
except one restricts the sum to odd values of $x$ and $y$, respectively. By combining our
results with the relations for the Hauptmoduli in Table~\ref{tab:j known},
it is possible to deduce many potentially interesting relations between classical Eisenstein
series $E_{k}(z)$, eta functions and theta functions.

For example, let us take $N=17$. In the notation of Theorem~\ref{thm:generators} one has
$M_{17}=9$ and the Hauptmodul $j_{17}(z)$ is given as a rational function of the form
$$
j_{17}(z)= \frac{P_{17}(E_4^{(17)}, E_6^{(17)},E_8^{(17)},E_{10}^{(17)},
E_{12}^{(17)})}{Q_{17}(E_4^{(17)}, E_6^{(17)},E_8^{(17)},E_{10}^{(17)},E_{12}^{(17)})},
$$
where $P_{17}$ and $Q_{17}$ denote polynomials of degree 9 in five variables with integer
coefficients, where coefficients are non-zero only if
the sum of products of weights and corresponding degrees is equal to $36$.

In a sense, this result is a direct analogue of formula~\eqref{j-invariant}
expressing the classical $j$-invariant for $\PSL(2,\Z)$ in terms of
classical holomorphic Eisenstein series.

Furthermore, formula~\eqref{E_k, p proposit fla} implies that the Eisenstein series
$E_{2k}^{(17)}$, for $k=2,3,4,5,6$ may be expressed as a linear combination of dilations
of series $E_{2k}$, hence the function
$$
\left( \frac{\theta_x(\frac{1}{2}, 0, \frac{17}{2}) -\theta_y(\frac{1}{2}, 0, \frac{17}{2}) }{
2\eta(z)\eta(17z)}\right)^2
$$
is a rational function in the Eisenstein series
$E_4(z)$, $E_4(17z)$, $E_6(z)$, $E_6(17 z)$, $E_8(z)$, $E_8(17z)$, $E_{10}(z)$, $E_{10}(17 z)$,
$E_{12}(z)$ and $E_{12}(17 z)$ with integer coefficients.

Proceeding in a similar manner, for example when $N=29$ or $N=47$, we obtain other
relations between theta functions, eta functions and holomorphic Eisenstein series $E_{2k}$.

%

\subsection{Groups $\Gamma_0(N)^+$ of higher genus}

There are $38$ different square-free levels $N$ such that $X_{N}$ has genus one. Similarly,
there are $39$ and $31$ different square-free $N$ such that $X_{N}$ has genus two and three,
respectively. In \cite{JST2}, the authors studied the $q$-expansions for the corresponding
function fields, proving that each function field admits two generators with various
properties, such as minimal pole at infinity and integer coefficients. In particular, a
polynomial relation was computed for each pair of generators, thus
giving an algebraic equation for the corresponding projective curve. In future studies, we
plan to investigate the various properties of these elliptic (genus one) and hyperelliptic
(genus two) curves. There are a vast number of problems, both arithmetic and analytic, to be
considered given that one knows the uniformizing group, a projective equation, $q$-expansions,
and relations to holomorphic Eisenstein series.

\end{document}